\documentclass[a4paper,11pt]{article} 

\usepackage{amsthm}                   
\usepackage{amsmath}                  
\usepackage{amsfonts}                 
\usepackage{amssymb}                  
\usepackage{mathrsfs}                 
\usepackage{subfigure}                
\usepackage{bbm}                      
\usepackage[margin=0.75in]{geometry}
\usepackage{multicol}                 
\usepackage{setspace}									
\usepackage{enumerate}								
\usepackage{mathtools}								
\usepackage{stmaryrd}									
\usepackage{fancyhdr}									
\usepackage{lastpage}									
\usepackage{soul}                       

\usepackage{xcolor}
\usepackage{graphicx}

\usepackage{soul} 

\numberwithin{equation}{section}

\newtheoremstyle{mythm}{6pt}{6pt}{\itshape}{}{\bfseries}{:}{ }{} 
\newtheoremstyle{mydef}{6pt}{6pt}{}{}{\bfseries}{:}{ }{}         
\newtheoremstyle{myrem}{6pt}{6pt}{}{}{\scshape}{:}{ }{}          

\theoremstyle{mythm}
\newtheorem{thm}{Theorem}
\newtheorem{prop}[thm]{Proposition}
\newtheorem{lemma}[thm]{Lemma}

\numberwithin{thm}{section}

\theoremstyle{myrem}
\newtheorem{rem}[thm]{Remark}
\newtheorem{ex}[thm]{Example}

\theoremstyle{mydef}
\newtheorem{defi}[thm]{Definition}
\newtheorem{assu}{Assumption}

\renewenvironment{proof}{\par\medbreak\noindent\textit{Proof}:\hskip.5em\ignorespaces}{\hfill\qedsymbol\medbreak} 
\renewcommand{\qedsymbol}{\hfill\rule{2mm}{2mm}}                                                                  

\newcommand{\Hs}{\mathcal{H}}
\newcommand{\C}{\ensuremath{\mathbbm{C}}}     
\newcommand{\N}{\ensuremath{\mathbbm{N}}}     

\DeclareMathOperator{\ran}{ran}																														  
\DeclareMathOperator{\Ker}{ker}																														  

\DeclareMathOperator{\spann}{span}

\newcommand{\und}[1]{_{_{#1}}}
\newcommand{\psinv}{M^\#\und{W,V}}
\newcommand{\mult}{M_{m,\Phi,\Psi}}

\makeatletter
\def\imod#1{\allowbreak\mkern10mu({\operator@font mod}\,\,#1)} 
\makeatother

\makeatletter
\newcommand*{\itemequation}[3][]{%
  \item
  \begingroup
    \refstepcounter{equation}%
    \ifx\\#1\\%
    \else
      \label{#1}%
    \fi
    \sbox0{#2}
    \sbox2{$\displaystyle#3\m@th$}
    \sbox4{\@eqnnum}

    \noindent
    \copy0\hspace{0.5em}
    \copy2\hfill\copy4
    \par
  \endgroup
  \ignorespaces
}
\makeatother  

\everymath{\displaystyle}

\begin{document}

    \title{Oblique Pseudoinverses of Multipliers}
    \author{Jerielle V. Malonzo and Diana T. Stoeva}
    \date{\today}

    \maketitle

    
    \textbf{ABSTRACT:}

    Multipliers are operators on a Hilbert space that have been implicitly used in applications throughout the years. {Multipliers are defined via three sequences and are mainly used in signal processing and psychoacoustics among many others. They modify a signal in three steps: analysis, term-wise scalar multiplication, and synthesis. While the study of multipliers stemmed from applications, its theory has earned an equally important place in the literature {in the last few decades}. One of the main questions in the theory of multipliers is the invertibility of multipliers. It is known that multipliers are not always invertible, and it has been shown that under mild conditions, the inverse of an invertible multiplier can be expressed as another multiplier of a very specific structure.} In this work, { we extend such result to the setting of possibly noninvertible multipliers. We solved the challenges of determining an appropriate generalized inverse to replace the inverse, and determining a proper duality relationship of the sequences of the resulting multiplier. We showed that these are nontrivial tasks that required the intricacies of the theory of frames for subspaces.   }
   

    \section{Introduction}

    Let $(\Hs_1,\langle\cdot,\cdot\rangle_{\Hs_1})$ and $(\Hs_2,\langle\cdot,\cdot\rangle_{\Hs_2})$ be Hilbert spaces. Consider  sequences $\Psi:=(\psi_n)_{n=1}^\infty $ with elements from $\Hs_1$, $\Phi:=(\phi_n)_{n=1}^\infty $  with elements from $\Hs_2$, and {a} scalar sequence $m:=(m_n)_{n=1}^\infty $ with elements from $\C$. Multipliers are operators $\mult$ of the form
    \begin{equation*}
        \mult h = \sum_{n=1}^\infty  m_n\langle h,\psi_n\rangle_{\Hs_1}\phi_n,
    \end{equation*}
    defined for those $h\in\Hs_1$ for which the series above converges in $\Hs_2$. As it can be observed, the action involves three steps: analysis using $\Psi$, term-wise multiplication by the scalar sequence $m$ (called the \emph{symbol} or the \emph{weight} of the multiplier), and synthesis using the sequence $\Phi$. 
    
    Multiplies arise naturally from applications. In fact, the so-called Gabor multipliers have been  implicitly used in applications throughout the years. Gabor multipliers serve as a main tool in signal processing and are used to implement time variant filters \cite{MatzHlawatschLinearTimeFreqFiltersOnLineAlgoandApps}. The first systematic theoretical study on such multipliers, on the other hand, is done in \cite{FeichtingerNowakAFirstSurveyofGaborMultipliers}. The theory of multipliers is rich and is being studied continuously, with Balazs pioneering the systematic study of Bessel multipliers \cite{BalazsBasicDefinitionandPropertiesofBesselMultipliers}, extending results about Gabor multipliers from \cite{FeichtingerNowakAFirstSurveyofGaborMultipliers}.
    
    Multipliers are deeply connected to another important concept -- \emph{frames}. Frames  generalize orthonormal bases and were introduced by Duffin and Schaeffer in \cite{DuffinSchaefferAclassofnonharmonicFourierSeries} in the context of nonharmonic Fourier series. They were then revived by Daubechies, Grossmann, and Meyer \cite{Daubechies_Grossmann_Meyer_1986} paving the way to today's recognition of frames in both theory and applications. We refer the readers to \cite{CasazzaTheArtofFrameTheory,  OleIntroductiontoFramesandRieszBases, HeilABasisTheoryPrimer, CasazzaKutyniokFiniteFrames} for further reading on the theory of frames. Among the many areas in which frames are applied today, we mention, for example, data compression \cite{ GoyalEtAlQuantizedFrameExpansionswithErasures}, compressed sensing \cite{FoucartRauhutAMathematicalIntroductiontoCompressiveSensing}, and sampling theory \cite{BenedettoFerreiraModernSamplingTheory}. 
    
    The reason frames are very useful in real-world applications is mainly due to their ability to provide {perfect and stable} reconstruction of a signal via what is called their \emph{dual frames}. It is known that a frame always has a dual frame. In particular, the so-called \emph{canonical dual frame} is constructed using a specific operator associated to the given frame, called the \emph{frame operator}. The \emph{frame operator} is positive, self-adjoint, and more importantly, invertible. It consists of an analysis step followed by a synthesis step, both with respect to the given frame. Hence, multipliers can be viewed as a generalization of frame operators obtained by inserting a multiplication step with a given sequence of weights. However, unlike frame operators, multipliers are not guaranteed to be invertible. 
    
    It is then natural to ask when else a multiplier is invertible. This question has been explored in several papers of Stoeva and Balazs. In \cite{StoevaBalazsInvertibilityofMultipliers, StoevaBalazsDetailedCharacterizationofConditionsfortheUnconditionalConvergenceandInvertibilityofMultipliers, StoevaBalazsASurveyontheUnconditionalConvergenceandtheInvertibilityofFrameMultiplierswithImplementation}, the authors provided some sufficient conditions for invertibility of a multiplier and gave formulas for the inverse in the case when inversion is possible. Interestingly, they showed in \cite{BalazsStoevaRepresentationofhtInverseofaFrameMultiplier} that the inverse of an invertible frame multiplier is a frame multiplier with a special structure. More precisely, {they showed the following results}.

    \begin{thm}\cite{BalazsStoevaRepresentationofhtInverseofaFrameMultiplier}\label{InvofMultasMultorig}
       Let $\Phi$ and $\Psi$ be frames for a Hilbert space $\Hs$. Suppose the multiplier $\mult$ is invertible with symbol $m$ satisfying $0<\inf_n|m_n|\le\sup_n|m_n|<\infty.$ Then the following are true. 
       \begin{enumerate}[(i)]
           \item There exists a unique dual frame $\Phi^\dagger$ of $\Phi$ such that 
           \begin{equation}\label{invofMultwithSpecificStructure1}
               \mult^{-1}=M_{1/m,\Psi^d,\Phi^\dagger}\quad \text{ for any dual frame $\Psi^d$ of $\Psi$}.
           \end{equation}
            \item There exists a unique dual frame $\Psi^\dagger$ of $\Psi$ such that 
            \begin{equation}\label{invofMultwithSpecificStructure2}
                \mult^{-1}=M_{1/m,\Psi^\dagger,\Phi^d}\quad \text{ for any dual frame $\Phi^d$ of $\Phi$}.
            \end{equation}
       \end{enumerate}
   \end{thm}
    Follow-up discussions in \cite{StoevaBalazsOntheDualFrameInducedbyanInvertibleFrameMultiplier} provide more discoveries on the sequences $\Phi^\dagger$ and $\Psi^\dagger$ used in Theorem \ref{InvofMultasMultorig}.


    \subsection{Motivation, goals and challenges} \label{GoalsandPlans}

    {In applications, multipliers are often noninvertible. Thus, it is of great interest, both in theory and applications, to investigate how one deals with noninvertible multipliers.}
    {\color{red} } Particularly, our main goal in this paper is to extend the setting in \cite{BalazsStoevaRepresentationofhtInverseofaFrameMultiplier, StoevaBalazsOntheDualFrameInducedbyanInvertibleFrameMultiplier},  particularly the identities in \eqref{invofMultwithSpecificStructure1} and \eqref{invofMultwithSpecificStructure2}, to possibly noninvertible  multipliers.    
    This is not an easy task and as we will see later, we will need more than trivial extensions of the steps behind the motivating results. This is because extending \eqref{invofMultwithSpecificStructure1} and \eqref{invofMultwithSpecificStructure2} requires completing two tasks simultaneously: determining an appropriate type of a generalized inverse and determining a type of ``frame duality" condition that would be satisfied.   Finding answers to these questions, opens the door to many possibilities. Aside from further enriching the theory of invertibility of multipliers, this is very promising in real world as different generalized inverses are used in different applications. 

    Stoeva and Taub\"ock were the first to tackle these challenges in \cite{StoevaTauboeck} and showed that one can express the \textit{Moore-Penrose inverse} (sometimes also simply called the \textit{pseudo-inverse}) of a frame multiplier as another multiplier, with reciprocal symbol and with   the analysis and synthesis sequences satisfying frame duality relationship. We take one step further and consider a more extensive generalized inverse and a more general type of frame duality -- the oblique pseudoinverses and the oblique dual frames. 

    Suppose we are given a linear mapping $M:\Hs_1\to\Hs_2$ with closed range, and consider closed subspaces $W$ and $V$ satisfying $\Hs_1=W\oplus\Ker(M)$ and $\Hs_2=\ran(M)\oplus V^\perp.$ The \textit{oblique pseudoinverse} is uniquely defined for each pair of subspaces $V$ and $W$. As noted in \cite{WardetalANoteontheObliqueMatrixPseudoinverse}, such operator is considered in different occasions in the literature.  It was first considered by Robinson \cite{RobinsonOntheGeneralizedInverseofanArbitraryLinearTransformation} in 1962. Calling it then the \emph{generalized inverse with respect to a particular decomposition}, he considered arbitrary linear transformations between finite-dimensional spaces. Five years later, Langenhop \cite{LangenhopOnGeneralizedInversesofMatrices} elaborated on such generalized inverses in the setting of matrices, and showed its close relations to the Moore-Penrose inverse. Then in 1968, Milne  \cite{MilneOMP1968} discussed the same generalized inverse with different but equivalent definition, and called it the \emph{oblique matrix pseudoinverse}. The oblique matrix pseudoinverse also coincides with another generalized inverse introduced by Chipman \cite{ChipmanOnLeastSquareswithInsufficientObservations} in 1964, called the \emph{weighted pseudoinverse} of a matrix. The equivalence of these two inverses is proven by Ward, Boullion, and Lewis in 1971 \cite{WardetalANoteontheObliqueMatrixPseudoinverse}. 
    
    The extension of Milne's oblique matrix pseudoinverses to the case where $\Hs_1$ and $\Hs_2$ may be infinite-dimensional is first considered by Eldar in 2001 \cite{EldarQuantumSignalProcessing}. Together with Werther in \cite{EldarTobiasGeneralFrameworkforConsistentSamplinginHilbertSpaces}, they {redefined such generalized inverses by additionally requiring the given operator to have closed range}. However, this extension of Milne's definition to possibly infinite-dimensional Hilbert spaces is also independently and briefly considered in \cite[Remark 3.8]{CorachMaestripieriWeightedGeneralizedInversesObliqueProjectionsandLeast-SquaresProblems}, where it was noted to coincide with the \textit{algebraic generalized inverse} introduced by Nashed and Votruba \cite{NashedVotrubaAUnifiedOperatorTheoryofGeneralizedInverses} in 1976.

    \textit{Oblique duality}, on the other hand, is introduced by Eldar in \cite{EldarSamplingwithArbitrarySamplingandReconstructionSpacesandObliqueDualFrameVectors} stemming from the motivation to allow more flexibility in the choice of the subspaces where the analysis and synthesis sequences belong. Together with Christensen in \cite{ChristensenEldarObliqueDualFramesandShiftInvariantSpaces}, they continued the discussion of oblique duals with focus on shift-invariant spaces. Since then, oblique duality has been studied extensively in the literature extending concepts in classical duality to the oblique case. Some examples are the extensions of the dual fusion frames and dual fusion frame systems in \cite{HeinekenMorillasObliqueDualFusionFrames}, the approximate duals in \cite{DiazHeinekenMorillasApproximateobliqueDualFrames, DiazHeinekenMorillasObliqueDulalityforFusionFrames}, and the dual frame completion in \cite{XiaoZhuZengODFinFinDimHS, KooLimExtensionofBesselSeqtoODFSeqeunces, AceskaMalonzoVelascoODFcompletion}. 

    \subsection{Organization}

    The paper is organized as follows: in Section \ref{NotationsandPrelims}, we set our notations and provide well-known basic results and definitions to be used throughout the paper. In Section \ref{Obl Psinv sec}, we focus on oblique pseudoinverses, giving an equivalent definition of such operators. Our new but equivalent definition characterizes the action of the oblique pseudoinverses more directly, and will prove to be more convenient for our purposes later on. We then provide in Section \ref{framesforSubsSec} new results on frames for subspaces, that are instrumental in proving the results in the remaining sections. Section \ref{OblPsinvofMultSec} will feature the main results of this paper, where we will prove that the oblique pseudoinverse of a multiplier is indeed another multiplier of a special structure. We discuss more about the determined analysis and synthesis sequences in the resulting multiplier, in Section \ref{OntheframesPhisharpandPsisharpSec}. We continue with Section \ref{UsingtheCanonicalOblDualsSec} by exploring when the oblique pseudoinverse will utilize the canonical oblique dual. Finally, in Appendix \ref{closednessofranM}  we present the different scenarios when a multiplier has closed range, which is a vital requirement for the well-definedness of the oblique pseudoinverse.

    
    \section{Notations and Preliminaries} \label{NotationsandPrelims}

    Throughout this paper, $\Hs,\Hs_1,$ and $\Hs_2$ will denote separable Hilbert spaces. The corresponding inner products will always be denoted by $\langle\cdot,\cdot\rangle$, omitting the subscript of the appropriate Hilbert space as it will be clear in the context. By $(x_n)_{n=1}^\infty \subseteq\Hs$, we mean that $x_n\in\Hs$ for all $n\in\N$. We write $\Psi$ and $\Phi$ for sequences  $(\psi_n)_{n=1}^\infty $ and $(\phi_n)_{n=1}^\infty $ with elements from a given Hilbert space, respectively, while we write $m$ for a scalar sequence $(m_n)_{n=1}^\infty \subseteq\C$. A linear mapping $M:\Hs_1\to\Hs_2$ will be simply called an \textit{operator}, and it is said to be \emph{invertible} if there is a bounded operator $M^{-1}:\Hs_2\to\Hs_1$ satisfying $MM^{-1}=I_{\Hs_2}$ and $M^{-1}M=I_{\Hs_1}$, where $I_{\Hs_1}$ and $I_{\Hs_2}$ are the identity operators on $\Hs_1$ and $\Hs_2$, respectively. The range and the kernel of an operator $M$ will be denoted by $\ran(M)$ and $\Ker(M),$ respectively. For a closed subspace $W$ of $\Hs$, the orthogonal projection of $\Hs$ onto $W$ is denoted by $\pi_W$. If $\ran(M)$ is closed, then there exists a unique and bounded operator $M^\dagger:\Hs_2\to\Hs_1$ satisfying $MM^\dagger x=x$ for all $x\in\ran(M)$. This operator $M^\dagger$ is called the  \textit{Moore-Penrose inverse} or simply \emph{pseudo-inverse} of $M$ (see, e.g., \cite[Section 2.5]{OleIntroductiontoFramesandRieszBases}).

    \subsection{Frames and Multipliers}

    We discuss in the first part of this subsection basic facts and results on frames (see, e.g.,   \cite{OleIntroductiontoFramesandRieszBases}). 
    
    A sequence $\Psi\subseteq\Hs$ is called a \emph{frame for $\Hs$} \cite{DuffinSchaefferAclassofnonharmonicFourierSeries} if there are positive constants $A$ and $B$ such that
    \begin{equation}\label{frameineq}
        A\|h\|^2\le\sum_{n=1}^\infty|\langle h,\psi_n\rangle|^2\le B\|h\|^2, \quad\quad\text{for every }\,h\in \Hs.
    \end{equation}
      The constants $A$ and $B$ are called a \emph{lower} and an \emph{upper frame bound,} respectively. Sometimes they are  denoted by $A_\Psi$ and $B_\Psi$, to highlight the frame $\Psi$.  If at least the right inequality of \eqref{frameineq} is satisfied, then $\Psi$ is called a \emph{Bessel sequence in $\Hs$} with \textit{Bessel bound} $B$. 

    Throughout the paper, we will mostly consider frames for closed subspaces of $\Hs.$ Note that a Bessel sequence in a closed subspace of $\Hs$ is necessarily a Bessel sequence in $\Hs.$ This guarantees the well-definedness of some relevant operators that will be used later. 
    \begin{lemma}\label{BesselinsubspacemeansBesselinspace}
   		Let $W$ be a closed subspace of $\Hs$ and let $\Psi$ be a Bessel sequence in $W$ with Bessel bound $B_\Psi$. Then $\Psi$ is a Bessel sequence in $\Hs$ with the same Bessel bound $B_\Psi.$
    \end{lemma}
	\begin{proof}
		Let $f\in\Hs$. Then $f=f_1+f_2$ for some unique vectors $f_1\in W$ and $f_2\in W^\perp$. Since $\Psi$ consists of vectors in $W$ and $f_2\in W^\perp$, we obtain that
		\begin{align*}
			\sum_{n=1}^\infty  |\langle f,\psi_n\rangle|^2=\sum_{n=1}^\infty  |\langle f_1,\psi_n\rangle|^2\le B_\Psi\|f_1\|^2\le  B_\Psi\|f\|^2,
		\end{align*}
        as desired.
	\end{proof}
    
    If one considers an orthonormal basis for $\Hs$, then it is immediate to see that \eqref{frameineq} is satisfied with $A=1=B$. That is, frames generalize orthonormal bases. But contrary to bases, frames allow redundancy, which is very useful in certain applications, for instance, signal processing. In particular, frames may provide non-unique expansions of a signal via the so-called \emph{dual frames}.

    \begin{defi}
        Let $\Psi$ be a frame for  $\Hs$. Then a frame $\Phi$ for $\Hs$ is called a \emph{dual frame of $\Psi$} if 
        \begin{equation}\label{dualframerecons}
            h = \sum_{n=1}^\infty \langle h,\psi_n\rangle\phi_n=\sum_{n=1}^\infty \langle h,\phi_n\rangle\psi_n,\quad \quad\forall\,h\in \Hs.
        \end{equation}
    \end{defi}
    Equivalent definitions for a frame and a dual frame can be obtained via the so-called synthesis and analysis operators. Given a Bessel sequence $\Psi$ in $\Hs$, the operator 
    \begin{equation}\label{synthesisopdefn}
        T_\Psi:(c_n)_{n=1}^\infty \mapsto \sum_{n=1}^\infty  c_n\psi_n
    \end{equation}
    is well-defined from $\ell^2(\N)$ into $\Hs$ and it is called the \emph{synthesis operator} of $\Psi.$ Its adjoint $T^\ast_\Psi:\Hs\to\ell^2(\N)$ is actually given by 
    \begin{equation}
        T^\ast_\Psi:h\mapsto(\langle h,\psi_n\rangle)_{n=1}^\infty ,
    \end{equation}
    and it is called the \emph{analysis operator of $\Psi$}.  The composition $S_\Psi:=T_\Psi T^\ast_\Psi$ is called the \emph{frame operator} of the Bessel sequence $\Psi$. 
    
    It is well-known that the operator $T_\Psi$ determined by \eqref{synthesisopdefn} is well-defined (hence, bounded) from $\ell^2(\N)$ into $\Hs$ if and only if $\Psi$ is a Bessel sequence in $\Hs$. Moreover, a Bessel sequence $\Psi$ is a frame for $\Hs$ if and only if $\ran(T_\Psi)=\Hs$, if and only if $S_\Psi$ is invertible on $\Hs$. 

    Let $\Psi$ be a frame for $\Hs$. From \eqref{dualframerecons}, $\Phi$ is a dual frame of $\Psi$ if and only if $T_\Phi T^*_\Psi=T_\Psi T^*_\Phi=I_\Hs$. Note that $(S_\Psi^{-1}\psi_n)_{n=1}^\infty $ is always a dual frame of $\Psi$ and it is called the \emph{canonical dual} of $\Psi.$ If removing one element from $\Psi$ causes the remaining sequence to no longer be a frame for $\Hs$, then $\Psi$ is called a \emph{Riesz basis}\footnote{Riesz bases were introduced in a different but equivalent way by Nina Bari in \cite{BariFrenchPaper,BariBiorthogonalSystemsandBasesinHilbertspaces}} for $\Hs$. Otherwise, $\Psi$ is called a \emph{redundant} or an \emph{overcomplete frame}. A Riesz basis has the canonical dual frame as its only dual frame, while overcomplete frames have infinitely many dual frames in addition to the canonical dual.

    Consider a Bessel sequence $\Psi$ in $\Hs_1$ and a vector $h\in\Hs_1$. Modifying the scalars $T^*\und\Psi h=(\langle h,\psi_n\rangle)_{n=1}^\infty $ via multiplication with a scalar sequence, and using a different sequence for the synthesis, one can extend the frame operator to another important class of operators -- the so-called {multipliers}. Given sequences $\Psi\subseteq\Hs_1, \,\Phi\subseteq\Hs_2,$ and $m\subseteq\C$, the \textit{multiplier} $\mult$ is the operator   determined by
    \begin{equation}\label{multdef}
        \mult h = \sum_{n=1}^\infty m_n\langle h,\psi_n\rangle\phi_n,
    \end{equation}
    for those $h\in\Hs_1$ such that the series in \eqref{multdef} converges in $\Hs_2$. Here, $m$ is called the \emph{symbol} or the \emph{weight} of the multiplier, while $\Psi$ and $\Phi$ are called the \emph{analysis sequence } and the \emph{synthesis sequence} of the multiplier, respectively. Given a scalar sequence $m$,  $\mathcal{M}_m$ will denote the \textit{multiplication operator} given by $\mathcal M_m(c_n)_{n=1}^\infty =(m_nc_n)_{n=1}^\infty $. 
    
     If $\Psi$ and $\Phi$ are Bessel sequences in (resp., frames for) $\Hs_1$ and $\Hs_2$, respectively, then $\mult$ is called a \emph{Bessel multiplier} (resp., \emph{frame multiplier}), and sometimes  we will use the representation $\mult=
    T_\Phi\mathcal{M}_mT^*_\Psi$.  In the next sections, we will often consider bounded symbols $m$ ($\sup_{n} |m_n|<L$ for some $L>0$) or \textit{semi-normalized} symbols $m$ ($\alpha\le\inf_{n}|m_n|\le\sup_{n} |m_n|\le \beta<\infty$ for some positive $\alpha$ and $\beta$).   If $\Hs_1=\Hs_2$, $m=(1)$, $\Psi=\Phi$, and $\Psi$ is  a frame for the entire space $\Hs_1$,  then the multiplier reduces to the frame operator $S_\Psi$ of $\Psi$.

    \subsection{Oblique projections and oblique dual frames}

    As mentioned in Section \ref{GoalsandPlans}, we will utilize oblique dual frames. A key ingredient in defining oblique dual frames is the concept of oblique projection. Let $W$ and $Z$ be closed subspaces of $\Hs$ satisfying $\Hs=W\oplus Z.$ In most literature (see, e.g., \cite{BehrensScharfSignalProcessingApplicationsofObliqueProjectionOperators, ChristensenEldarObliqueDualFramesandShiftInvariantSpaces, OleIntroductiontoFramesandRieszBases}), the oblique projection onto $W$ along $Z$ is defined as the mapping that acts on $W$ as the identity, and acts on $Z$ as the zero map. In this paper, we opt to use a known equivalent definition, that will be consistent with our presentation of the oblique pseudoinverses in Section \ref{Obl Psinv sec}.

     \begin{defi}\label{oblprojmydef}
        Let $W$ and $Z$ be closed subspaces of $\Hs$ such that $\Hs=~W\oplus Z.$ Then every $h\in\Hs$ can be written uniquely as $h=h_1+h_2$, where $h_1\in W$ and $h_2\in\ Z$. Thus the mapping $\pi\und{W,Z}:h\mapsto h_1$ is well-defined from $\Hs$ onto $W$,
        and it is called the \emph{oblique projection onto $W$ along $Z$}.
    \end{defi}
    
    As noted earlier, $\ker(\pi\und{W,Z})=Z$ and $\ran(\pi\und{W,Z})=W.$ We collect and list below some properties of the oblique projection that will be used in the rest of the paper. {Proofs can be found in \cite[Section III.5.4]{KatoPerturbationTheoryforLinearOperators}, \cite[Lemma 2.1]{HeinekenMorillasObliqueDualFusionFrames} for (i) and (ii), respectively, while (iii) is straightforward from \cite[Theorem 2.3]{TangObliqueProjectionsBiorthogonalRieszBasesAndMultiwaveletsInHilbertSpaces}}.

    \begin{lemma}\label{propoblproj}
        Let $W$ and $Z$ be closed subspaces of $\Hs$ such that $\Hs=W\oplus Z.$ Then the following statements are true.
        \begin{enumerate}[(i)]
            \item $\pi\und{W,Z}:\Hs\to\Hs$ is bounded.
            \item $\pi\und{Z^\perp}\pi\und{W,Z}=\pi\und{Z^\perp}$ and $\pi\und{W,Z}\pi\und{Z^\perp}=\pi\und{
            W,Z}$.
            \item $\Hs=Z^\perp\oplus W^\perp$ and $\pi\und{W,Z}^*=\pi\und{Z^\perp,W^\perp}$.
        \end{enumerate}
    \end{lemma}

    It is important to note that the oblique projection $\pi\und{W,Z}$ reduces to the orthogonal projection onto $W$, when $W=Z^\perp$. Further, an oblique projection is not always self-adjoint and in fact, is self-adjoint if and only if it is an orthogonal projection. More relations between the two types of projections are investigated in \cite{GrevilleSolutionoftheMatrixEqnandRelationsbetweenObliqueandorthogonalProjectors}.
    
    Below we recall the definition of an oblique dual frame. Using oblique dual frames, one can have the flexibility of choosing different subspaces where the analysis and synthesis would take place.

    \begin{defi}
        Let $W_1$ and $W_2$ be closed subspaces of $\Hs$ such that $\Hs=W_1\oplus W_2^\perp$.  Let $\Psi$ and $\Phi$ be frames for $W_1$ and $W_2$, respectively.  Then $\Psi$ is called an \emph{oblique dual frame of $\Phi$ on $W_1$} if $T_\Psi T_\Phi^*=\pi\und{W_1,W_2^\perp}$. 
    \end{defi}

    \begin{rem}
        The definition above is symmetric. Indeed, if $\Psi$ is an oblique dual frame of $\Phi$ on $W_1$, then by Lemma \ref{propoblproj} we have $\Hs=W_2\oplus W_1^\perp$ and $T_\Phi T^*_\Psi=(T_\Psi T_\Phi^*)^*=(\pi\und{W_1,W_2^\perp})^*=\pi\und{W_2,W_1^\perp}$, which means that $\Phi$ is an oblique dual frame of $\Psi$ on $W_2.$ Thus, we can also refer to $\Psi$ and $\Phi$ collectively as an \emph{oblique dual frame pair}. If $W_1=W_2$, then the oblique dual frame pair is a dual frame pair in the subspace $W_1$.
    \end{rem}

    Extending a result from the classical frame theory where duality relations of Bessel sequences lead to frames, to the oblique case, the following holds.

    \begin{lemma}\label{dualityrelationshipmeansframes}
        Let $W_1$ and $W_2$ be closed subspaces of $\Hs$ such that $\Hs=W_1\oplus W\und2^\perp$. 
        Let $\Psi$ and $\Phi$ be Bessel sequences in $W_1$ and $W_2$, respectively. If $T\und\Psi T^*\und{\Phi}=\pi\und{W_1,W_2^\perp}$, then $\Psi$ and $\Phi$ are frames for $W_1$ and $W_2$, respectively. 
    \end{lemma}
    \begin{proof}
        {By Lemma \ref{BesselinsubspacemeansBesselinspace}, both $\Psi$ and $\Phi$ are Bessel sequences in $\Hs$ and so, $T_\Psi T^*_\Phi$ is a well-defined operator from $\Hs$ into $\Hs$}. Assume $T\und\Psi T^*\und{\Phi}=\pi\und{W_1,W_2^\perp}$ on $\Hs.$ Then $W_2=\ran(\pi^*\und{W_1,W_2^\perp})=\ran(T_\Phi T^*\und\Psi)\subseteq\ran(T_\Phi)\subseteq W_2$, implying that $T_\Phi$ maps $\ell^2(\N)$ onto $W_2$, which shows that $\Phi$ is a frame for $W_2.$ Furthermore, $T_\Phi T^*\und\Psi=\pi\und{{W_2,W_1^\perp}}$, which implies in a similar way as above that $\Psi$ is a frame for $W_1.$
    \end{proof}

   A given frame always has what is so-called the \emph{canonical oblique dual frame}. {This sequence was first introduced in \cite{EldarSamplingwithArbitrarySamplingandReconstructionSpacesandObliqueDualFrameVectors, EldarTobiasGeneralFrameworkforConsistentSamplinginHilbertSpaces}, where it was referred to as the (lone) oblique dual frame. The family of oblique dual frames known today, of which the canonical oblique dual frame is a part, was subsequently defined  in \cite{ChristensenEldarObliqueDualFramesandShiftInvariantSpaces}.} 
    \begin{defi}
        Let $W_1$ and $W_2$ be closed subspaces of $\Hs$ such that $\Hs=W_1\oplus W_2^\perp.$ Let $\Psi$ be a frame for $W_1$. The sequence $\pi\und
        {W_2,W_1^\perp}S^\dagger\und{\Psi}\Psi$ is called the \emph{canonical oblique dual of $\Phi$ on $W_2.$}
    \end{defi}
    
    The canonical oblique dual frame plays a similar role to the canonical dual frame -- every vector $f\in W_1$ can be reconstructed via analysis  using the canonical oblique dual frame and synthesis using the given frame. {Another similar role is that} the sequence of coefficients obtained from the analysis using the canonical oblique dual  provides the minimum $\ell^2-$norm among all possible coefficients in the reconstruction of $f$ (using synthesis of the given frame). This result was first proved in \cite{EldarSamplingwithArbitrarySamplingandReconstructionSpacesandObliqueDualFrameVectors} in the setting of finite-dimensional Hilbert spaces, and was proved to be true also in the infinite-dimensional setting in \cite{EldarTobiasGeneralFrameworkforConsistentSamplinginHilbertSpaces}.


    \section{Oblique pseudoinverses}\label{Obl Psinv sec}
	
	We split this section into two parts: first, we present select versions of the definitions of oblique pseudoinverses in literature and second, we give our own equivalent definition, that will be more convenient for our work. We fix our setting. 

    \begin{assu}\label{mainassum}
	      Let $M:\Hs_1\to\Hs_2$ be a bounded operator whose range is closed. Consider closed subspaces $W$ of $\Hs_1$ and $V$ of $\Hs_2$ such that
	\begin{align}
	    \Hs_1&=W\oplus \ker(M) \label{H1space}
        \\ \Hs_2&=\ran(M)\oplus V \label{H2space}.
	\end{align} 
	\end{assu}

    We start with the definition of the oblique pseudoinverse in the finite-dimensional setting due to Milne.

    \begin{defi}\cite{MilneOMP1968}\label{oblpsinvMilnedef}
        Let $\Hs_1$ and $\Hs_2$ be finite-dimensional inner product spaces over the field of complex numbers and let $M:\Hs_1\to\Hs_2$ be an operator. Suppose we have decompositions \eqref{H1space} and \eqref{H2space}, where $W$ is isomorphic to $\Ker(M)^\perp$. Then the \emph{oblique matrix pseudoinverse of $M$ with respect to the subspaces $V,W$} is the homomorphism $M^\dagger\und{W,V}:\Hs_2\to\Hs_1$ satisfying the following:
        \begin{enumerate}[(i)]
            \item $M^\dagger\und{W,V}M x = x$ for all $x\in W$,
            \item $M^\dagger\und{W,V} y = 0$ for all $y\in V$.
        \end{enumerate}
    \end{defi}

    \begin{rem}
        The oblique psuedoinverse is defined with respect to the subspaces $W$ and $V$. Note that the decompositions \eqref{H1space} and \eqref{H2space} are not necessarily orthogonal. Thus, the oblique pseudoinverse can be defined for a larger class of subspaces $W$ and $V$ satisfying \eqref{H1space} and \eqref{H2space}.  Interestingly, the class of oblique pseudoinverses over all possible subspaces $W$ and $V$ contains the Moore-Penrose inverse. Indeed, the Moore-Penrose inverse is obtained by taking $W=\Ker(M)^\perp$ and $V=\ran(M)^\perp.$
    \end{rem}   

    Oblique pseudoinverses in the setting of infinite-dimensional spaces were first considered by Eldar in her dissertation \cite{EldarQuantumSignalProcessing}, but we follow the convention in her follow-up paper with Werther \cite{EldarTobiasGeneralFrameworkforConsistentSamplinginHilbertSpaces}.
    
    \begin{defi}\label{oblpsinvYoninadef}\cite{EldarTobiasGeneralFrameworkforConsistentSamplinginHilbertSpaces}
		Suppose Assumption \ref{mainassum} holds. The \emph{oblique pseudoinverse of $M$ on $W$ along $V$} is the unique mapping $\psinv:\Hs_2\to\Hs_1$ satisfying
		\begin{align}
			M\psinv  &= \pi\und{\ran(M),V}  \label{ops cond1}\\
			\psinv M&=\pi\und{W,\ker(M)} \label{ops cond2} \\
			\ran(\psinv)&=W. \label{ops cond3}
		\end{align}
	\end{defi}          

    We now give our own definition for the oblique pseudoinverse (Definition \ref{oblpsinvmydef}), in a similar spirit as in the presentation of the oblique projection in Definition \ref{oblprojmydef}. For this purpose, first we make the following observation.

    \begin{lemma}\label{Lemmaforourpsinvdef}
        Suppose Assumption \ref{mainassum} holds. Then for any $y\in\Hs_2$, there exist unique vectors $x_y\in W$ and $z_y\in V$ such that $y=Mx_y +z_y.$
    \end{lemma}
    \begin{proof}
        Let $y\in\Hs_2$. From \eqref{H2space}, there exist unique vectors $y'\in\ran(M)$ and $z_y\in V$ such that $y=y'+z_y.$ Since $y'\in\ran(M)$, we must have $y'=Mx_y$ for some $x_y\in W.$ It is left to show that $x_y$ is unique. Suppose there exists another $x\in W$ such that $y=Mx+z_y$. Then $Mx_y+z_y=Mx+z_y$ and so $x_y-x\in\Ker(M).$ So, $x_y-x\in\Ker(M)\cap W$ and by \eqref{H1space}, we have $x_y=x$.
    \end{proof} 

    We are now ready to present our proposed definition of the oblique pseudoinverses.
    
    \begin{defi}\label{oblpsinvmydef}
        Suppose Assumption 1 holds. The \emph{oblique pseudoinverse of $M$ on $W$ along $V$} is the mapping $\psinv$ {determined on $\Hs_2$ as follows: for $y\in\Hs_2$, put $\psinv y:=x_y$}, where $x_y$ is the unique element in $W$ determined in Lemma \ref{Lemmaforourpsinvdef}.
    \end{defi}

    Aside from being a constructive definition, it is also easier to see from  our presentation of the oblique pseudoinverse that it extends the oblique matrix pseudoinverse of Milne in Definition \ref{oblpsinvMilnedef}. Indeed, it follows straightforward from Definition \ref{oblpsinvmydef} that $\psinv Mx=x$ for all $x\in W$ and $\psinv y =0$ for all $y\in V$.

    We now prove that our definition of the oblique pseudoinverse is equivalent to the definition by Eldar and Werther in Definition \ref{oblpsinvYoninadef}.

    \begin{prop}\label{EldarDefnandOursareEquivalent}
        Definition \ref{oblpsinvYoninadef}  and Definition \ref{oblpsinvmydef} are equivalent.
    \end{prop}
    \begin{proof}
      We first assume $\psinv$ is defined as in Definition \ref{oblpsinvYoninadef}. Let $y\in\Hs_2$. From \eqref{ops cond1} and Lemma \ref{Lemmaforourpsinvdef}, we have
    \begin{align*}
        M\psinv y =  \pi\und{\ran(M),V}y= \pi\und{\ran(M),V}(Mx_y+z_y)=Mx_y,
    \end{align*}
    where $x_y$ and $z_y$ are the elements determined in Lemma \eqref{Lemmaforourpsinvdef}.
    {Using} \eqref{ops cond2}, it then follows that
    \begin{align*}
        \pi\und{W,\ker(M)}\psinv  y =\psinv (M\psinv y) = \psinv M x_y=\pi\und{W,\Ker(M)} x_y = x_y.
    \end{align*}
    But condition \eqref{ops cond3} implies that $\psinv y = \pi\und{W,\ker(M)}\psinv y.$ In particular, we see here that $\psinv y=x_y,$ as defined in Definition \ref{oblpsinvmydef}. 

    Conversely, let $\psinv$ be given as in Definition \ref{oblpsinvmydef}. We now show that $\psinv$  also satisfies conditions \eqref{ops cond1}-\eqref{ops cond3} in Definition \ref{oblpsinvYoninadef}. Clearly, we have $\ran(\psinv)=W.$ Let $y\in\Hs_2$. Considering the vectors $x_y$ and $z_y$ determined in Lemma \ref{Lemmaforourpsinvdef}, we obtain
        \begin{align*}
            M(\psinv y) = Mx_y = \pi\und{\ran(M),V}(Mx_y+z_y) = \pi\und{\ran(M),V}y,
        \end{align*}
        showing condition \eqref{ops cond1}. Finally, for any $u\in\Hs_1$, we can write $u=u_1+u_2$ uniquely {with} $u_1\in W$ and $u_2\in\Ker(M)$, and thus,
        \begin{align*}
            \psinv M u =\psinv M (u_1+u_2) = \psinv(Mu_1)=u_1=\pi\und{W,\Ker(M)}u,
        \end{align*}
        which proves condition \eqref{ops cond2}.
    \end{proof}
    
    Thereafter, we assume Assumption \ref{mainassum} and we use Definition \ref{oblpsinvmydef} for the oblique pseudoinverse.  {We end this section by listing some basic properties of the oblique pseudoinverse that we will use later. These properties can be proven straightforward using Definition \ref{oblpsinvmydef}.
    \begin{prop}\label{propertiesofoblpsinvs}
        Under Assumption \ref{mainassum}, the following hold for the oblique pseudoinverse $\psinv$:
        \begin{enumerate}[(i)]
            \item $\psinv$ is a linear and bounded mapping from $\Hs_2$ into $\Hs_1$.
            \item  $\psinv M\psinv = \psinv$ and $M\psinv M = M$.
        \end{enumerate}
    \end{prop}}


    \section{On Frames for Subspaces} \label{framesforSubsSec}

    While frames have been thoroughly discussed through the years, frames for subspaces have not been explored as much. In this section, we highlight some properties of a frame for a subspace that up to our knowledge, has never been shed light on in the literature. We start by looking at the range and kernel of the frame operator and its Moore-Penrose inverse. 
    \begin{prop}\label{ranandkerofpseudoinvofframeop}
        Let $W$ be a closed subspace of $\Hs$ and let $\Psi$ be a frame for $W$ (hence, a Bessel sequence in $\Hs)$. Consider the frame operator $S_\Psi:\Hs\to\Hs$. Then the following statements hold.
        \begin{enumerate}[(i)]
            \item $\ran(S_\Psi)=W$ and so $S_\Psi^\dagger$ is well-defined. Moreover, $\ran(S^\dagger_\Psi)=W$;
            \item $\ker(S_\Psi)=W^\perp=\ker(S^\dagger\und\Psi).$
        \end{enumerate}
    \end{prop}
    \begin{proof}
        We will first prove that $\ker(S_\Psi)=W^\perp$. Observe that $\Ker(S_\Psi)=\Ker(T_\Psi T^\ast_\Psi)\supseteq\Ker(T^\ast_\Psi)=\ran(T_\Psi)^\perp=W^\perp$, since $\Psi$ being a frame for $W$ implies that $\ran(T_\Psi)=W.$ Now we can express any $x\in\Hs$  as a sum $\pi\und W x+\pi\und{W^\perp}x$ and we have
        \begin{align*}
            S_\Psi x=S_\Psi (\pi\und W x+\pi\und{W^\perp}x)=S_\Psi (\pi\und Wx).
        \end{align*}
        Thus, $S_\Psi x=0$ iff $S_\Psi(\pi\und W x)=0.$ But $\Psi$ being a frame for $W$ implies that $S_\Psi|\und W$ is injective and so, $\Ker(S_\Psi|\und W)=\{0\}.$ Therefore, $S_\Psi x=0$  iff $\pi\und{W}x=0$ iff $x\in W^\perp,$ completing the proof that $\Ker(S_\Psi)=W^\perp.$

        Now from $\Hs=W\oplus W^\perp$ and from $\Ker(S_\Psi)=W^\perp$, it follows that $\ran(S_\Psi)=\ran(S_\Psi|\und W)$, and since $S_\Psi|\und W$ is invertible on $W$, we have $\ran(S_\Psi)=W$. Hence, $S^\dagger\und\Psi$ is well-defined on $\Hs$.  The rest of the results follow from the facts that $S_\Psi$ is self-adjoint, $\ran(S^*\und\Psi)=\ran(S^\dagger\und\Psi)$, and $\Ker(S^*\und\Psi)=\Ker(S^\dagger\und\Psi)$ (see e.g. \cite[Section 2.5]{OleIntroductiontoFramesandRieszBases}).
    \end{proof}

    We prove next that a frame for a subspace is uniquely determined by the set of its oblique dual frames. This is an extension of what was shown in \cite{BalazsStoevaRepresentationofhtInverseofaFrameMultiplier} that a frame for the whole space is uniquely determined by its set of dual frames.

    \begin{prop}\label{framecharacterizedbyitssetofODF}
        Let $\Hs$ be a separable Hilbert space and let $W_1$ and $W_2$ be closed subspaces of $\Hs$ satisfying $\Hs=W_1\oplus W_2^\perp$. Let $\Phi$ be a frame for $W_2$ and let $\mathcal{P}(\Phi)$ denote the collection of all oblique dual frames of $\Phi$ on $W_1$. Then the following statements hold.
        \begin{enumerate}[(i)]
            \item $\displaystyle\overline{\bigcup_{\Phi^\text{od}\in\mathcal{P}(\Phi)} \ran(T^\ast_{\Phi^\text{od}})}=\ell^2.$
            \item Let $G$ be a frame for $W_2$. If every oblique dual frame of $\Phi$ on $W_1$ is also an oblique dual frame of $G$ on $W_1$, then $\Phi=G.$
        \end{enumerate}
    \end{prop}
    \begin{proof}
        The proof  follows closely the idea of the proof of \cite[Theorem 1.2]{BalazsStoevaRepresentationofhtInverseofaFrameMultiplier}.
        \begin{enumerate}[(i)]
            \item Let $c=(c_n)_{n=1}^\infty \in\left(\bigcup_{\Phi^\text{od}\in\mathcal{P}(\Phi)}\ran(T^\ast_{\Phi^\text{od}})\right)^\perp$. Then  $\langle T_{\Phi^\text{od}}c,h\rangle=\langle c,T^\ast_{\phi^\text{od}}h\rangle=0$ for every $\Phi^\text{od}\in\mathcal{P}(\Phi)$ and for every $h\in\Hs.$ As a result, we have 
            \begin{equation}\label{Tcis0}
                T\und{\Phi^\text{od}}c=0, \quad\forall\,\Phi^\text{od}\in\mathcal{P}(\Phi).
            \end{equation}
            From \cite[Theorem 3.2]{ChristensenEldarObliqueDualFramesandShiftInvariantSpaces}, every oblique dual frame of $\Phi$ on $W_1$ is of the form
            \begin{equation}\label{formforobliquedualframes}
                \Phi^\text{od}=\left(\pi\und{W_1,W_2}S_\Phi^\dagger\phi_n+\theta_n-\sum_{j=1}^\infty\langle S_\Phi^\dagger\phi_n,\phi_j\rangle\theta_j\right)_{n=1}^\infty 
            \end{equation}
            where $\Theta=(\theta_n)_{n=1}^\infty $ is a Bessel sequence in $W_1$.
            Since $\phi_j=\pi\und{W_2,W_1^\perp}\phi_j=\pi\und{W_1,W_2^\perp}^\ast\phi_j$ for all $j\in\N$, we obtain an equivalent representation of \eqref{formforobliquedualframes} given by
            \begin{equation*}
                \Phi^\text{od}=\left(\widetilde\phi_n+\theta_n-\sum_{j=1}^\infty\langle \widetilde\phi_n,\phi_j\rangle\theta_j\right)_{n=1}^\infty ,
            \end{equation*}
            where $\widetilde\phi_n=\pi\und{W_1,W_2^\perp}S_\Phi^\dagger\phi_n$ is the canonical oblique dual frame of $\Phi$ on $W_1$. Using \eqref{Tcis0}, we obtain
            \begin{equation}\label{Tcis0again}
                0=T_{\widetilde\Phi}c+\sum_{n=1}^\infty c_n\left(\theta_n-\sum_{j=1}^\infty \langle\widetilde\phi_n,\phi_j\rangle\theta_j\right)=\sum_{n=1}^\infty c_n\theta_n-\sum_{n=1}^\infty\sum_{j=1}^\infty c_n\langle\widetilde\phi_n,\phi_j\rangle\theta_j
            \end{equation}
            for any Bessel sequence $\Theta$ in $W_1.$ For every $k\in\N$, take the Bessel sequence $\Theta^{(k)}=(\delta_{n,k}e_n)_{n=1}^\infty $, where $(e_n)_{n=1}^\infty $ is an  orthonormal basis of the closed subspace $W_1$. Then for every $k\in\N$ we have from \eqref{Tcis0again} that
            \begin{align*}
                0=c_ke_k-\sum_{n=1}^\infty c_n\langle\widetilde\phi_n,\phi_k\rangle e_k=c_ke_k-\left\langle\sum_{n=1}^\infty c_n\widetilde\phi_n,\phi_k\right\rangle e_k\stackrel{\eqref{Tcis0}}{=}c_ke_k.
            \end{align*}
            That is, $c_k=0$ for all $k\in\N$, which gives us the desired conclusion.

            \item Assume every oblique dual frame $\Phi^\text{od}$ of $\Phi$ on $W_1$ is also an oblique dual frame of $G$ on $W_1$, that is, $T_G T^\ast_{\Phi^\textbf{od}}=\pi\und{W_2,W_1^\perp}=T_\Phi T^\ast_{\Phi^\textbf{od}}$ for all $\Phi^{\text{od}}\in\mathcal{P}(\Phi).$ From (i), we conclude that $T_G=T_\Phi$, from which it immediately follows that $G=\Phi.$
        \end{enumerate}
    \end{proof}

    \begin{rem}\label{framecharacterizedbyitssetofODFver2}
        Let $\Hs$ be a separable Hilbert space and let $W_1$ and $W_2$ be closed subspaces of $\Hs.$ As noted already in Lemma \ref{propoblproj}, the decomposition $\Hs=W_1\oplus W_2^\perp$ is equivalent to the decomposition $\Hs=W_2\oplus W_1^\perp$, and thus, one gets a ``symmetric'' version of Proposition \ref{framecharacterizedbyitssetofODF} by interchanging the roles of the subspaces $W_1$ and $W_2$.
    \end{rem}


    \section{Oblique Pseudoinverse of Multipliers}\label{OblPsinvofMultSec}
   
    In this section, our goal is to extend Theorem \ref{InvofMultasMultorig} to possibly non-invertible multipliers, and express their oblique pseudoinverse as another multiplier of a similar desired structure. We first fix our setting.

    \begin{assu}\label{assumponmult}
        Let $\Psi$ and $\Phi$ be Bessel sequences in $\Hs_1$ and $\Hs_2,$ respectively, and consider a bounded scalar sequence $m$. Then the multiplier $\mult:\Hs_1\to\Hs_2$ is a well-defined bounded operator and the convergence in \eqref{multdef} is unconditional (by \cite[Theorem 6.1]{BalazsBasicDefinitionandPropertiesofBesselMultipliers}). Assume further that $\ran(\mult)$ is closed. Let $W\subseteq\Hs_1$ and $V\subseteq\Hs_2$ be closed subspaces that satisfy \eqref{H1space} and \eqref{H2space}, respectively. Put $M:=\mult$ and consider the oblique pseudoinverse $\psinv$ of $M$ on $W$ along $V.$
    \end{assu}

    \begin{rem}\label{multiplierwithnonclosedrange}
        In order to consider the oblique pseudoinverse of $\mult$, {we have to assume that $\mult$ has closed range since it is not guaranteed from the assumptions alone on the sequences $\Psi,\Phi,$ and $m$ in Assumption \ref{assumponmult}}. As an example, consider $\Hs=\ell^2(\N)$, $m=(1/n)$, and $\Psi=\Phi=(\delta_n)$, where $(\delta_n)_{n=1}^\infty $ is the canonical basis for $\ell^2(\N)$. With these, we obtain a Bessel multiplier $\mult$ with bounded symbol $m$, but with non-closed range. Indeed, for $x=(x_n)_{n=1}^\infty \in\Hs$, 
        \[ M_{m,\Phi,\Psi} x=\left(\frac{x_n}{n}\right)_{n=1}^\infty . \]
        Now for each $k\in\N$, define the sequence $x^{(k)}\in\Hs$ via
        \[x^{(k)}_n = \begin{cases}
            1, &\text{if }n\le k
            \\0, &\text{if } n>k
        \end{cases}.\]
        Then $Mx^k$ converges to $\left(\frac1n\right)_{n=1}^\infty $ in $\Hs$, but $\left(\frac1n\right)_{n=1}^\infty \not\in\ran(M)$.

        In Appendix \ref{closednessofranM}, we will discuss more classes of multipliers that have closed or non-closed range.
    \end{rem} 
    
      The first task in our extension of Theorem \ref{InvofMultasMultorig} is to find  appropriate sequences that would replace the sequences $\Psi^\dagger$ and $\Phi^\dagger$. These special frames were determined {in Theorem \ref{InvofMultasMultorig}} as follows:
    \begin{itemize}
        \item $\Psi^\dagger=(\mult^{-1}(m_n\phi_n))_{n=1}^\infty $,
        \item $\Phi^\dagger=((\mult^{-1})^\ast(\overline{m_n}\psi_n))_{n=1}^\infty $.
    \end{itemize}
    As we aim to find an expression for the oblique pseudoinverse, the natural thing to do is to replace the inverse of $M$ with an appropriate oblique pseudoinverse, and consider instead 

    \begin{itemize}
        \itemequation[psisharp]{}{\Psi^\#\und{W,V}:=(\psinv(m_n\phi_n))_{n=1}^\infty } 
        \itemequation[phisharp]{}{\Phi^\#\und{W,V}:=((\psinv)^\ast(\overline{m_n}\psi_n))_{n=1}^\infty }
    \end{itemize}
        
    However, being in a more general setting than the one in \cite{BalazsStoevaRepresentationofhtInverseofaFrameMultiplier}, we face several hurdles along the way.

    {First}, the new sequences $\Psi^\#\und{W,V}$ and $\Phi^\#\und{W,V}$ are no longer guaranteed to be frames for the mother spaces. What we can only see from initial inspection is
    \[ \Psi\und{W,V}^\#\subseteq\ran(\psinv)=W \quad\text{and}\quad \Phi^\#\und{W,V}\subseteq\ran((\psinv)^\ast)=(\Ker(\psinv))^\perp=V^\perp.\]
    {Second}, while in Theorem \ref{InvofMultasMultorig}, $\Psi^\dagger$ and $\Phi^\dagger$ are dual frames of $\Psi$ and $\Phi$, respectively, we can immediately say that the same relationship might fail to hold in our setting with $\Psi^\#\und{W,V}$ and $\Phi^\#\und{W,V}$. Indeed, under Assumption \ref{assumponmult},  $\Psi$ and $\Phi$ are not even assumed to be frames (and even if they were assumed to be frames, we go back to the first problem -- the frame property of $\Psi^\#\und{W,V}$ and $\Phi^\#\und{W,V}$). Not assuming the frame properties of $\Psi$ and $\Phi$ gives us our third obstacle: if we use oblique duality instead of the classical duality (as we have been hinting from the beginning), we have to replace $\Psi$ and $\Phi$ with proper frames that could be oblique dual frames of $\Psi^\#\und{W,V}$ and $\Phi^\#\und{W,V}$.

    From now on, under Assumption 2,  $\Psi^\#_{W,V}$ and $\Phi^\#_{W,V}$ are always given as in \eqref{psisharp} and \eqref{phisharp}, respectively. For simplicity, we use $\Psi^\#:=\Psi^\#\und{W,V}$ and $\Phi^\#:=\Phi^\#\und{W,V}$, if $V$ and $W$ are clear in the context.
    
    We face our first challenge and show that while $\Psi^\#\und{W,V}$ and $\Phi^\#\und{W,V}$ are not necessarily frames for $\Hs_1$ and $\Hs_2,$ respectively, they turn out to be frames for $W$ and $V^\perp$, respectively.

    \begin{thm}\label{PsisharpandPhisharpareframes}
        Suppose Assumption \ref{assumponmult} holds. Then the following statements are true.
        \begin{enumerate}[(i)]
            \item $\Psi^\#_{W,V}$ is a frame for $W$.
            \item $\Phi^\#_{W,V}$ is a frame for $V^\perp.$
        \end{enumerate}
    \end{thm}
    \begin{proof}
        \begin{enumerate}[(i)]
            \item  {Under Assumption \ref{assumponmult}, $m\Phi$ is a Bessel sequence in $\Hs_2$ and since $\psinv$ is bounded, $\Psi^\#$ is a Bessel sequence in $\Hs_1$. Thus, $T\und{\Psi^\#}$ is well defined and bounded on $\ell^2(\N)$ and furthermore, we have $T_{\Psi^\#}=\psinv T\und{m\Phi}$. It is thus clear that $\ran(T_{\Psi^\#})\subseteq W$. We now show that $\ran(T\und{\Psi^\#})=W$. Let $h\in W$. Then
            \begin{align*}
                h = \psinv M h=\psinv T_{m\Phi}T^*_\Psi h,
            \end{align*}
            where $T^*_\Psi h\in\ell^2(\N)$ and so $h\in\ran(\psinv T_{m\Phi})=\ran(T_{\Psi^\#}).$ Therefore, $\ran(T_{\Psi^\#})=W$, which completes the proof that $\Psi^\#$ is a frame for $W.$}

            \item
            {Similar to above, $\overline{m}\Psi$ is a Bessel sequence in $\Hs_1$, $T\und{\Phi^\#}$ is well defined and bounded on $\ell^2(\N)$, and $
            T\und{\Phi^\#}=(\psinv)^*T\und{\overline{m}\Psi}$.
            Since it is now clear that $\ran(T\und{\Phi^\#})\subseteq V^\perp$, we just need to show the converse inclusion. Using
            Lemma \ref{propoblproj} and the identity \eqref{ops cond1}, for every $h\in V^\perp$ we can write 
            $h=(\pi\und{\ran(M),V})^*h = (M\psinv)^*h=(\psinv)^* T\und{\overline{m}\Psi}T^*_\Phi h= T_{\Phi^\#}T^*_\Phi h.$
            Therefore, $V^\perp\subseteq \ran( T\und{\Phi^\#}  )$.}
        \end{enumerate}
    \end{proof}

    We solve the other two hurdles in the next result. The challenge is to relate $\Psi^\#$ and $\Phi^\#$ with $\Psi$ and $\Phi$, respectively. We surmise that they would satisfy some {oblique} dual relationship. Since decompositions \eqref{H1space} and \eqref{H2space} hold, and using the above theorem, our objective is to find  frames for $\ker(M)^\perp$ and $\ran(M)$ that would be oblique dual frames of $\Psi^\#$ and $\Phi^\#$, respectively. The first sequences that come to mind are $\Psi$ and $\Phi$, but this would be careless as they are not necessarily frames {for $\Ker(M)^\perp$ and $\ran(M)$, and not even necessarily contained in $\Ker(M)^\perp$ and $\ran(M)$, respectively}. As a remedy, we restrict $\Psi$ and $\Phi$ using appropriate oblique pseudoinverses and we obtain the following result.
    \begin{thm}\label{partnersofPhiandPsisharps}
        Suppose Assumption \ref{assumponmult} holds. Then the following hold.
        \begin{enumerate}[(i)]
            \item $\pi\und{W,\Ker(M)}^*\Psi$ is a frame for $\Ker(M)^\perp$.
            \item $\pi\und{\ran(M),V}\Phi$ is a frame for $\ran(M)$.
            \item $\pi^*\und{W,\Ker(M)}\Psi$ and $\Psi^\#\und{W,V}$ are a pair of oblique dual frames.
            \item $\pi\und{\ran(M),V}\Phi$ and $\Phi^\#\und{W,V}$ are a pair of oblique dual frames.
        \end{enumerate}
    \end{thm}
    \begin{proof} For simplicity, write $\Psi^K:=\pi^*\und{W,\Ker(M)}\Psi$ and $\Phi^R:=\pi\und{\ran(M),V}\Phi$. 
        \begin{enumerate}[(i)]
            \item Clearly, $\pi^*\und{W,\Ker(M)}\Psi=\pi\und{\Ker(M)^\perp,W^\perp}\Psi\subseteq\Ker(M)^\perp.$ Since $\Psi$ is Bessel in $\Hs_1$ and $\pi^*\und{W,\Ker(M)}$ is bounded on $\Hs_1$, we have that $\pi^*\und{W,\Ker(M)}\Psi$ is a Bessel sequence in $\Ker(M)^\perp.$  Then $T_{\Psi^K}$ is well-defined and bounded with $\ran(T_{\Psi^K})\subseteq\Ker(M)^\perp.$ We will prove the reverse inclusion to complete the proof. Let $h\in\Ker(M)^\perp=\ran(M^*)$. Then there exists $h_2\in\Hs_2$ such that $h=M^*h_2.$ Moreover, we have $h=\pi^*\und{W,\ker(M)}h$. Therefore,
            \begin{align*}
                h=\pi^*\und{W,\Ker(M)}M^*h_2 = \pi^*\und{W,\Ker(M)}T_\Psi T_{m\Phi}^*h_2 = T_{\Psi^K}(T^*_{m\Phi}h_2),
            \end{align*}
            which is exactly what we want.

            \item As {above}, $T_{\Phi^R}$ is well-defined and bounded on $\ell^2(\N)$ with $\ran(T_{\Phi^R})\subseteq\ran(M).$ We show the converse inclusion to complete the proof. If $h\in\ran(M)$, then $h=\pi\und{\ran(M),V}h$ and $h=Mh_1$ for some $h_1\in\Hs_1.$ Thus,
            \begin{align*}
                h=\pi\und{\ran(M),V}Mh_1=\pi\und{\ran(M),V} T_\Phi T^*_{\overline m\Psi}h_1 = T_{\Phi^R} (T^*_{\overline m\Psi}h_1).
            \end{align*}

            \item From (i) and Theorem \ref{PsisharpandPhisharpareframes}(i), it is left to show that $T_{\Psi^\#}T^*_{\Psi^K}=\pi\und{W,\Ker(M)}.$ Indeed, for any $h\in\Hs_1$ we have
            \begin{align*}
                T_{\Psi^\#}T^*_{\Psi^K}h &=\sum_{n=1}^\infty \langle h, \pi^*\und{W,\Ker(M)}\psi_n\rangle \psinv(m_n\phi_n)
                = \psinv \left( \sum_{n=1}^\infty  m_n\langle \pi\und{W,\Ker(M)} h, \psi_n\rangle \phi_n\right)
                \\&= \psinv M(\pi\und{W,\Ker(M)}h)
                = \pi\und{W,\Ker(M)}h.
            \end{align*}

            \item From (ii) and Theorem \ref{PsisharpandPhisharpareframes}(ii), it remains to show that $T_{\Phi^R}T^*_{\Phi^\#}=\pi\und{\ran(M),V}.$  Indeed, for any $h\in\Hs_2$, we obtain
            \begin{align*}
                T_{\Phi^R}T^*_{\Phi^\#} h &= \sum_{n=1}^\infty  \langle h, (\psinv)^*(\overline{m_n}\psi_n)\rangle \pi\und{\ran(M),V}\phi_n
                = \pi\und{\ran(M),V}M\psinv h
                = \pi\und{\ran(M),V}h.
            \end{align*}
        \end{enumerate}
    \end{proof}

   We now have all the necessary ingredients to extend Theorem \ref{InvofMultasMultorig} to possibly noninvertible multipliers considering oblique pseudoinverses. 
    \begin{thm}\label{oblpsinvofmultprop}
        Suppose Assumption \ref{assumponmult} holds and suppose in addition that $m$ is semi-normalized.  Then the following hold.
        \begin{enumerate}[(i)]
            \item $\psinv=M_{\tfrac{1}{m},\Psi\und{W,V}^\#,\Phi^{\text{od}}}$ for any oblique dual frame $\Phi^\text{od}$ of $\pi\und{\ran(M),V}\Phi$ on $V^\perp$.
            \item $\psinv=M_{\tfrac{1}{m},\Psi^{\text{od}},\Phi^{\#}\und{W,V}}$ for any oblique dual frame $\Psi^\text{od}$ of $\pi^\ast\und{W,\Ker(M)}\Psi$ on $W$.
            \item $\psinv=M_{\tfrac{1}{m},\Psi\und{W,V}^\#,\Phi\und{W,V}^\#}$.
           \end{enumerate}
    \end{thm}
    \begin{proof} 
    For convenience, as in the proof of Theorem \ref{partnersofPhiandPsisharps}, we denote $\Psi^K:=\pi^\ast\und{W,\Ker(M)}\Psi$ and $\Phi^R:=\pi\und{\ran(M),V}\Phi$. From Theorem \ref{partnersofPhiandPsisharps}, it is shown that $\Psi^K$ and $\Phi^R$ are frames for $\Ker(M)^\perp$ and $\ran(M)$, respectively.
        \begin{enumerate}[(i)]
            \item Let $h\in\Hs_2$ and let $\Phi^\text{od}$ be any oblique dual frame of $\Phi^R$ on $V^\perp$. Then
            \begin{align*}
                M_{\tfrac{1}{m},\Psi^\#,\Phi^{\text{od}}}h &= \sum_{n=1}^\infty  \frac{1}{m_n}\langle h, (\Phi^{\text{od}})_n\rangle \psinv (m_n\phi_n)
                = \psinv \left(\sum_{n=1}^\infty  \langle h, (\Phi^{\text{od}})_n\rangle \phi_n \right).
            \end{align*}
            {Using Proposition \ref{propertiesofoblpsinvs}(ii) and property \eqref{ops cond1}}, we have 
            \begin{align*}
                 \psinv \left(\sum_{n=1}^\infty  \langle h, (\Phi^{\text{od}})_n\rangle \phi_n \right)&=\psinv M \psinv \left(\sum_{n=1}^\infty  \langle h, (\Phi^{\text{od}})_n\rangle \phi_n \right)
                \\&= \psinv \left(\sum_{n=1}^\infty  \langle h, (\Phi^{\text{od}})_n\rangle \pi\und{\ran(M),V}\phi_n \right) 
                =\psinv T\und{\Phi^R}T^\ast\und{\Phi^\text{od}}h
                \\&=\psinv \pi\und{\ran(M),V}h
                = \psinv M\psinv h
                =\psinv h.
            \end{align*}

            \item Let $h\in\Hs_2$ and let $\Psi^\text{od}$ be any oblique dual frame of $\Psi^K$ on $W$. {Using Proposition \ref{propertiesofoblpsinvs}(ii) and \eqref{ops cond2}, we get}
            \begin{align*}
                M_{\tfrac{1}{m},\Psi^{\text{od}},\Phi^{\#}} h &=\sum_{n=1}^\infty \frac{1}{m_n}\langle h, (\psinv)^*(\overline{m_n}\psi_n)\rangle (\Psi^\text{od})_n
                = \sum_{n=1}^\infty  \langle \psinv M\psinv h, \psi_n\rangle (\Psi^\text{od})_n
                \\&= \sum_{n=1}^\infty  \langle\psinv h, \pi\und{W,\Ker(M)}^*\psi_n\rangle (\Psi^\text{od})_n 
                = T\und{\Psi^\text{od}}T^*\und{\Psi^K}(\psinv h)
                = \pi\und{W,\Ker(M)}\psinv h
                = \psinv h.
            \end{align*}
            \item {Follows from (i) and Theorem \ref{partnersofPhiandPsisharps}(iv)}.    
        \end{enumerate}
    \end{proof}

    {We now illustrate the previous theorem in the next example.}

    \begin{ex}
         Let $(e_n)_{n=1}^\infty $ be an orthonormal basis of $\Hs$. Take $m=(1)$, $\Psi=(e_{n+2})_{n=1}^\infty $ and $\Phi=(e_{n+1})_{n=1}^\infty $. Clearly,
        $M:=\mult$ is not invertible on $\Hs$ and
        \begin{itemize}
            \item $\Ker(M)=\spann\{e_1,e_2\}$,
            \item $\ran(M)=\overline{\spann\{e_{n+1}\}_{n=1}^\infty }$.
        \end{itemize}
        With the choice of subspaces $W=\overline{\spann\{e_2+e_3,e_4,e_5,\dots\}}$ and $V=\spann\{e_1+e_2\}$, we satisfy $\Hs=W\oplus\Ker(M)=\ran(M)\oplus V$. The oblique pseudoinverse $\psinv$ of $M$ on $W$ along $V$ is well-defined and is given by
        \begin{equation}
            \psinv y= \langle y, e_2-e_1\rangle (e_2+e_3)+\sum_{n=4}^\infty \langle y, e_{n-1}\rangle e_{n}
        \end{equation}
        for all $y\in\Hs$. Furthermore, 
        \[ \psinv = M_{1/m,\Psi^\#,\Phi^{\text{od}}}, \]
        where $\Psi\und{W,V}^\#=(e_2+e_3,e_4,e_5,\dots)$ and $\Phi^{\text{od}}=(e_2-e_1,e_3,e_4,\dots)$ is an oblique dual of $\Phi^R=\Phi$ on $V^\perp$.
     \end{ex}


    \section{On the frames $\Phi^\#\und{W,V}$ and $\Psi^\#\und{W,V}$}\label{OntheframesPhisharpandPsisharpSec}

    Balazs and Stoeva gave in \cite[Proposition 3.1]{BalazsStoevaRepresentationofhtInverseofaFrameMultiplier} additional  properties of the determined special frames $\Psi^\dagger$ and $\Phi^\dagger$ in Theorem \ref{InvofMultasMultorig}, in particular, that they are unique {in some sense}. We also extend these properties to our setting and obtain the following result.

    \begin{prop}
        Suppose Assumption \ref{assumponmult} holds and suppose $m$ is a semi-normalized scalar sequence. Then the following statements hold. 
        \begin{enumerate}[(i)]
            \item If $F$ is a Bessel sequence in $V^\perp$ such that $\psinv=M_{\tfrac1m,\Psi^\#,F}$, then $F$ is an oblique dual frame of $\pi\und{\ran(M),V}\Phi$ on $V^\perp$. 
            \item If $G$ is a Bessel sequence in $W$ such that $\psinv=M_{\tfrac1m, G,\Phi^\#}$, then $G$ is an oblique dual frame of $\pi^\ast\und{W,\Ker(M)}\Psi$ on $W.$
            \item The sequence $\Psi^\#$ is the only Bessel sequence in $W$ satisfying Theorem \ref{oblpsinvofmultprop}(i). 
            \item The sequence $\Phi^\#$ is the only Bessel sequence in $V^\perp$ satisfying Theorem \ref{oblpsinvofmultprop}(ii).
        \end{enumerate}
    \end{prop}
    \begin{proof}
        \begin{enumerate}[(i)]
            \item  Let $F$ be a Bessel sequence in $V^\perp$ such that $\psinv=M_{\tfrac1m,\Psi^\#,F}$. Having in mind Lemma \ref{dualityrelationshipmeansframes}, it is enough to prove that $T\und{\Phi^R}T^*\und F=\pi\und{\ran(M),V}$, where $\Phi^R$ denotes $\pi\und{\ran(M),V}\Phi$. We have
            $$\begin{gathered}
                \pi\und{\ran(M),V}\stackrel{\text{\eqref{ops cond1}}}{=}M\psinv\stackrel{\text{hypothesis}}{=}MM_{\tfrac1m,\Psi^\#,F}=MT\und{\Psi^\#}T^*\und{\frac{1}{\overline m}F}
                \\\stackrel{\text{\eqref{psisharp}}}{=}M\psinv T\und{m\Phi}T^*\und{\frac{1}{\overline m}F}
                \stackrel{\eqref{ops cond1}}{=}\pi\und{\ran(M),V}T_\Phi T^*_F=T\und{\Phi^R}T^*\und F.
            \end{gathered}$$

            \item Let $G$ be a Bessel sequence in $W$ such that $\psinv=M\und{\frac{1}{m},G,\Phi^\#}$. From Lemma \ref{dualityrelationshipmeansframes}, it suffices to show that $T\und GT^*_{\Psi^K}=\pi\und{W,\Ker(M)}$, where $\Psi^K$ denotes $\pi^\ast\und{W,\Ker(M)}\Psi$. Now,
            \begin{equation*}
                \begin{gathered}
                    \pi\und{W,\Ker(M)}\stackrel{\eqref{ops cond2}}{=}\psinv M
                    \stackrel{\text{hypothesis}}{=}M\und{\tfrac1m,G,\Phi^\#}M
                    =T\und GT^*\und{\tfrac1{\overline m}\Phi^\#}M
                    \\\stackrel{\eqref{phisharp}}{=}T\und G T^*\und\Psi\psinv M
                    \stackrel{\eqref{ops cond2}}{=}T\und G T^*\und\Psi\pi\und{W,\Ker(M)}
                    =T\und G T^*\und{\Psi^K}.                \end{gathered}
            \end{equation*}

            \item The proof follows closely the idea of the proof of \cite[Proposition 3.1]{BalazsStoevaRepresentationofhtInverseofaFrameMultiplier}. Denote by $\mathcal{M}_{1/m}$ the multiplication operator defined by $\mathcal{M}_{1/m}((c_n)_{n=1}^\infty )=(c_n/m_n)_{n=1}^\infty $ for any $(c_n)_{n=1}^\infty \in\ell^2(\N).$ Suppose $G$ is a Bessel sequence in $W$ satisfying $\psinv=M_{\tfrac{1}{m},G,\Phi^{\text{od}}}$ for every oblique dual $\Phi^\text{od}$ of $\pi\und{\ran(M),V}\Phi$ on $V^\perp$. Then for any oblique dual frame $\Phi^\text{od}$ of $\pi\und{\ran(M),V}\Phi$ on $V^\perp$, we have 
            \begin{align*}
                T_{\Psi^\#}\mathcal{M}_{\frac{1}{m}}T^*_{\Phi^\text{od}}=M_{\frac{1}{m},\Psi^\#,\Phi^\text{od}}=M_{\frac{1}{m},G,\Phi^\text{od}}=T_{G}\mathcal{M}_{\frac{1}{m}}T^*_{\Phi^\text{od}}.
            \end{align*}
            Using Remark \ref{framecharacterizedbyitssetofODFver2} and the invertibility of $\mathcal{M}_\frac{1}{m}$, we get $T_{\Psi^\#}=T_{G}$ and so $\Psi^\#=G$, as desired.
            
            \item Suppose $F$ is a Bessel sequence in $V^\perp$ such that $\psinv=M\und{\frac1m,\Psi^\text{od},F}$ for every oblique dual $\Psi^\text{od}$ of $\pi\und{W,\Ker(M)}^*\Psi$ on $W.$ Then for any oblique dual frame $\Psi^\text{od}$ of $\pi\und{W,\Ker(M)}^*\Psi$ on $W$, we obtain
            \[ T\und{\Phi^\#}\mathcal{M}^*_\frac1mT^*\und{\Psi^\text{od}}=M^*\und{\frac1m,\Psi^\text{od},\Phi^\#}=M\und{\frac1m,\Psi^\text{od},F}^*=T_F\mathcal{M}^*_\frac1mT^*\und{\Psi^\text{od}}. \]
            Using Proposition \ref{framecharacterizedbyitssetofODF} and the invertibility of $\mathcal{M}^*\und{\frac1m}$, we get $T\und{\Phi^\#}=T_F$ and so $\Phi^\#=F.$    
        \end{enumerate}
    \end{proof}


    \section{Using the Canonical Oblique Duals} \label{UsingtheCanonicalOblDualsSec}

     The main motivation to express the inverse of a multiplier as another multiplier of special type comes from what was shown in \cite{BalazsBasicDefinitionandPropertiesofBesselMultipliers} about multipliers whose analysis and synthesis sequences are Riesz bases, namely, if $m$ is semi-normalized and $\Psi$ and $\Phi$ are Riesz bases, then the multiplier $\mult$ is invertible and $\mult^{-1}=M\und{1/m,\widetilde\Psi,\widetilde\Phi}$, where $\widetilde\Psi$ and $\widetilde\Phi$ are the canonical duals of $\Psi$ and $\Phi$, respectively.
     
    Considering invertible frame multipliers for overcomplete frames, this result was extended in \cite{BalazsStoevaRepresentationofhtInverseofaFrameMultiplier} (see Theorem \ref{InvofMultasMultorig}). In Theorem \ref{oblpsinvofmultprop}, we provide more general results considering the oblique pseudoinverse instead of the inverse. In this section, we provide sufficient conditions when $\Psi^\#\und{W,V}$ and $\Phi^\#\und{W,V}$ in the representation in Theorem \ref{oblpsinvofmultprop}(iii) will be the canonical oblique duals of some appropriate frames.

    Under Assumption 2, define $\Psi^K=(\Psi_n^K)_{n\in\N}$ and $\Phi^R=(\Phi^R_n)_{n\in\N}$ as in the proof of Theorem \ref{partnersofPhiandPsisharps}, where it was shown that they are frames for $\Ker(M)^\perp$ and $\ran(M)$, respectively. Their respective canonical oblique dual frames are ~$\widetilde\Psi^K:=\pi\und{W,\Ker(M)}S^\dagger\und{\Psi^K}\Psi^K$ and $\widetilde\Phi^R:=\pi^\ast \und{\ran(M),V}S^\dagger\und{\Phi^R}\Phi^R$. 

    \begin{prop}
        Suppose Assumption \ref{assumponmult} holds and suppose further that $m$ is semi-normalized. {If there exists a bounded operator $U:\Hs_2\to\Hs_1$ such that $U\phi^R_n=\overline{m_n}\psi^K_n$ for all $n\in\N$ \footnote{This condition is equivalent to the condition $\ran(T^*\und{\overline m\Psi^K})\subseteq\ran(T^*\und{\Phi^R})$, see \cite{BalanEquivalenceRelationsandDistancesbetweenHilbertframes}.}, 
        then $\psinv=M\und{(1/m), \widetilde\Psi^K,\widetilde\Phi^R}$.
        }
    \end{prop}
    \begin{proof}
         Let $y\in\Hs_2$ and let $x_y\in W$ be as described in Lemma \ref{Lemmaforourpsinvdef}.
        Consider the multiplication operator $\mathcal{M}_{1/m}$, which clearly maps $\ell^2(\N)$ into $\ell^2(\N)$. Note that every frame operator is self-adjoint, and thus so is every pseudo-inverse of a frame operator. We have 
         \begin{align*}
            M\und{(1/m), \widetilde\Psi^K,\widetilde\Phi^R}y 
           & \stackrel{\text{defn. of $\widetilde\Phi^R$}}{=} T\und{\widetilde\Psi^K}\mathcal{M}_{1/m}T\und{\Phi^R}^\ast  S^\dagger\und{\Phi^R} \pi\und{\ran(M),V}y
            = T\und{\widetilde\Psi^K}\mathcal{M}_{1/m}T\und{\Phi^R}^\ast  S^\dagger\und{\Phi^R} \pi\und{\ran(M),V}Mx_y
            \\&=T\und{\widetilde\Psi^K}\mathcal{M}_{1/m}T\und{\Phi^R}^\ast S^\dagger\und{\Phi^R} \pi\und{\ran(M),V}T\und{\Phi}T^\ast\und{\overline m\Psi}x_y
             =T\und{\widetilde\Psi^K}\mathcal{M}_{1/m}T\und{\Phi^R}^\ast S^\dagger\und{\Phi^R} T\und{\Phi^R}T^\ast\und{\overline m\Psi}x_y
            \\&\stackrel{x_y\in W}{=}
            T\und{\widetilde\Psi^K}\mathcal{M}_{1/m}T\und{\Phi^R}^\ast S^\dagger\und{\Phi^R} T\und{\Phi^R}T^\ast\und{\overline m\Psi}\pi\und{W,\Ker(M)}x_y
            =T\und{\widetilde\Psi^K}\mathcal{M}_{1/m}T\und{\Phi^R}^\ast S^\dagger\und{\Phi^R} T\und{\Phi^R}T^\ast\und{\overline m\Psi^K}x_y.
        \end{align*}
        {From assumption (see the footnote), it follows  that  $\ran(T^\ast\und{\overline m\Psi^K})\subseteq\ran(T^\ast\und{\Phi^R})$}. Thus, there exists $y^x\in\Hs_2$ such that $T^\ast\und{\overline m\Psi^K}x_y=T^\ast\und{\Phi^R}y^x$. Applying Proposition \ref{ranandkerofpseudoinvofframeop}, we have ${\ran(S^\dagger\und{\Phi^R})}={\ran(M)}$, and so $S^\dagger\und{\Phi^R}S\und{\Phi^R}=\pi\und{\ran(S^\dagger\und{\Phi^R})}=\pi\und{\ran(M)}$. Thus, we have 
        \begin{align*}
            T\und{\widetilde\Psi^K}\mathcal{M}_{1/m}T\und{\Phi^R}^\ast  S^\dagger\und{\Phi^R} T\und{\Phi^R}T^\ast\und{\overline m\Psi^K}x_y
            &= T\und{\widetilde\Psi^K}\mathcal{M}_{1/m}T\und{\Phi^R}^\ast S^\dagger\und{\Phi^R} T\und{\Phi^R}T^\ast\und{\Phi^R}y^x
            {=} T\und{\widetilde\Psi^K}\mathcal{M}_{1/m}T\und{\Phi^R}^\ast  \pi\und{\ran(M)}y^x.
        \end{align*}
        Furthermore, using the fact that $\Phi^R$ is in $\ran(M)$, we get
        \begin{align*}
            T\und{\widetilde\Psi^K}\mathcal{M}_{1/m}T\und{\Phi^R}^\ast  \pi\und{\ran(M)}y^x
            &
            =T\und{\widetilde\Psi^K}\mathcal{M}_{1/m}( \pi\und{\ran(M)}T\und{\Phi^R})^\ast y^x
            =T\und{\widetilde\Psi^K}\mathcal{M}_{1/m}T\und{\Phi^R}^\ast y^x.
        \end{align*}
        Finally, using again $T^\ast\und{\overline m\Psi^K}x_y=T^\ast\und{\Phi^R}y^x$, 
         we obtain
        \begin{align*}
            T\und{\widetilde\Psi^K}\mathcal{M}_{1/ m}T\und{\Phi^R}^\ast  y^x
            =T\und{\widetilde\Psi^K}\mathcal{M}_{1/ m}T\und{\overline m\Psi^K}^\ast  x_y
            = T\und{\widetilde\Psi^K}T\und{\Psi^K}^\ast  x_y
            =x_y.
        \end{align*}
        Since $\psinv y=x_y$, the proof is complete.
    \end{proof}

    \appendix
    
    \section{On the closedness of the range of $\mult$} \label{closednessofranM}

As it is well known, every invertible operator 
has closed range. Unlike invertible multipliers, non-invertible multipliers need not always have closed range (see, for example, Remark \ref{multiplierwithnonclosedrange} and Example \ref{examplesofclosedandnonclosedrange}). 
In this section, we list different combinations of the analysis and synthesis sequences with respect to redundancy or non-redundancy, and investigate whether the resulting multiplier have closed or non-closed range. 
    \begin{prop}\label{whenisrangeclosedgivenoneRiesz}
        Let $\Phi$ be a Riesz sequence in $\Hs_2$ and let $m$ be bounded. Let $I:=\{n\in \N: m_n\neq0\}$ be nonempty.
        \begin{enumerate}[(i)]
            \item If $(\psi_n)_{n\in I}$ is a frame sequence in $\Hs_1$  and $(m_n)_{n\in I}$  is semi-normalized, then $\ran(\mult)$ and $\ran(M_{m,\Psi,\Phi})$ are  closed in $\Hs_1$ and $\Hs_2$, respectively. 
            \item If $\Psi$ is a Riesz sequence in $\Hs_1$, then $\ran(\mult)$ is closed in $\Hs_2$ if and only if $(m_n)_{n\in I}$ is  semi-normalized.
        \end{enumerate}   
    \end{prop}
    \begin{proof}
        Let $M:=\mult$ and put  $W_1:=\overline\spann\{\psi_n:n\in I\}$ and $W_2:=\overline{\spann}\{\phi_n : n\in I\}$.  Since $Mh=\sum_{n=1}^\infty m_n\langle h,\psi_n\rangle\phi_n=\sum_{n\in I}m_n\langle h,\psi_n\rangle\phi_n$ for all $h\in\Hs_1$, it is immediate that $W_1^\perp\subseteq\Ker(M)$ and $\ran(M)\subseteq W_2.$ Further, being closed subspaces of $\Hs_1$ and $\Hs_2$, respectively, both $W_1$ and $W_2$ are Hilbert spaces. We consider the operator $M_1:=M|_{W_1}:W_1\to W_2$, where $\ran(M_1)=\ran(M)$. Thus throughout, we investigate the closedness of $\ran(M_1)$.  
        \begin{enumerate}[(i)]
            \item Suppose $(\psi_n)_{n\in I}$ is a frame sequence in $\Hs_1$  and $(m_n)_{n\in I}$  is semi-normalized. In this case, both $(m_n\psi_n)_{n\in I}$ and $(\overline m_n\psi_n)_{n\in I}$ are still frame sequences. We show first that $\ran(M_1)$ is closed. Observe that $M_1$ can also be written as $T_{(\phi_n)_{n\in I}} T^*_{(\overline m_n\psi_n)_{n\in I}}$. Consider a  Cauchy sequence \\$(T_{(\phi_n)_{n\in I}} T^*_{(\overline m_n\psi_n)_{n\in I}}h_\ell)_{\ell\in I}$ in $\ran(M_1)$.  Fix $\varepsilon>0$. Since $\Phi$ is a Riesz sequence in $\Hs_2$, we have that $T_{(\phi_n)_{n\in I}}:\ell^2(I)\to W_2$ is invertible. Along with $(T_{(\phi_n)_{n\in I}} T^*_{(\overline m_n\psi_n)_{n\in I}}h_\ell)_{\ell\in I}$ being Cauchy, there must exist $K>0$ and $N>0$ such that
                \[ K\|T^*_{(\overline m_n\psi_n)_{n\in I}}(h_k-h_\ell)\|\le\|T_{(\phi_n)_{n\in I}} T^*_{(\overline m_n\psi_n)_{n\in I}}(h_k-h_\ell)\|<\varepsilon  \]
            for any $k,\ell\in I$ with $k,\ell\ge N.$ Hence, $(T^*_{(\overline m_n\psi_n)_{n\in I}}h_\ell)_{\ell\in I}$ is a Cauchy sequence in $\ell^2(I)$, hence, convergent to some $c\in\ell^2(I).$ But $(\overline m_n\psi_n)_{n\in I}$ is a frame sequence and so $\ran(T^*_{(\overline m_n\psi_n)_{n\in I}})$ is closed in $\ell^2(I)$. Therefore, $c=T^*_{(\overline m_n\psi_n)_{n\in I}}h$ for some $h\in\Hs_1.$ This shows that $\ran(M_1)$ is closed in $\Hs_2$.
            
            Now, for the closedness of $\ran(M_{m,\Psi,\Phi})$, denote $M_2:=T_{(m_n\psi_n)_{n\in I}}T^*_{(\phi_n)_{n\in I}}$ and thus, $\ran(M_{m,\Psi,\Phi})=\ran(M_2).$ Finally, $\ran(M_2)$ is closed since  $\ran(T^*_{(\phi_n)_{n\in I}})=\ell^2(I)$ and $\ran(T_{(m_n\psi_n)_{n\in I}})$ is closed in $\Hs_1$ (from  $(m_n\psi_n)_{n\in I}$ being a frame sequence).

            \item Suppose $\Psi$ is a Riesz sequence in $\Hs_1$. The ``if'' part follows from (i). For the other direction, assume that $(m_n)_{n\in I}$ is not semi-normalized. We shall show that $\ran(M_1)$ is not closed, by showing that $\ran(M_1)$ is dense in $W_2$ but is not equal to $W_2.$ Let $g\in\ran(M_1)^\perp$. Then for any $h\in W_1$, we obtain
            \begin{equation}\label{prooffornonclosed}
                0=\langle g, M_1 h\rangle =\left\langle g,\sum_{n\in I} m_n\langle h,\psi_n\rangle \phi_n \right\rangle =\sum_{n\in I}\langle g,\phi_n\rangle\overline{m_n\langle h, \psi_n\rangle}.
            \end{equation}
             Let $(\tilde\psi_n)_{n\in I}$ be the unique biorthogonal sequence of $(\psi_n)_{n\in I}$. For each $k\in I$, consider $h:=\tilde\psi_k/m_k$ in \eqref{prooffornonclosed} and conclude that
            \[ \langle g,\phi_k\rangle=0. \]
            This forces $g\in W_2^\perp$, which implies that $\ran(M_1)$ is dense in $W_2$. Recall that $M_1$ is injective since $(m_n\psi_n)_{n\in I}$ is a complete Bessel sequence in $W_1$ \cite[Proposition 3.1]{StoevaBalazsRieszBasesMultipliers}\footnote{\cite[Proposition 3.1]{StoevaBalazsRieszBasesMultipliers} and \cite[Proposition 5.1]{StoevaBalazsInvertibilityofMultipliers} concern the case of multipliers whose domain and codomain are the same Hilbert space, but the results naturally extend to when the domain and the codomain are possibly different}, but $M_1$ is not invertible since $(\psi_n)_{n\in I}\subseteq W_1$ and $(\phi_n)_{n\in I}\subseteq W_2$ are Riesz bases (for the respective Hilbert spaces $W_1$ and $W_2$) and $(m_n)_{n\in I}$ is not semi-normalized \cite[Theorem 5.1]{StoevaBalazsInvertibilityofMultipliers}\footnotemark[\value{footnote}]. Therefore, $\ran(M_1)\neq W_2=\overline{\ran(M_1)}.$

        \end{enumerate}
    \end{proof}

    We illustrate Proposition \ref{whenisrangeclosedgivenoneRiesz} through the following examples.
    \begin{ex}\label{examplesofclosedandnonclosedrange}
            \begin{enumerate}[(1)]
                Throughout, $\Hs$ is an arbitrary separable Hilbert space.  
                \item Let $(e_n)_{n=1}^\infty $ be an orthonormal basis for $\Hs$. Consider $m=(0,1,1,1,\dots)$, $\Psi=(e_2,e_2,e_3,e_4,e_5,e_6,\dots)$, and $\Phi=(e_1,e_2,e_3,e_4,e_5,\dots)$. Then $M:=\mult=M_{m,\Psi,\Phi}$ with closed range $\ran(M)=\overline{\spann}\{e_n : n\ge 2\}$. 
                
                \item 
                \begin{enumerate}
                    \item Consider $m = (0,1,1,1,\dots)$ and $\Psi$, $\Phi$ to be Riesz sequences in $\Hs$.  Thus,  we have that $\ran(\mult)=\overline{\spann}\{\phi_n\,:\,n\ge2\}$ and $\ran(M_{m,\Psi,\Phi})=\overline{\spann}\{\psi_n\,:\,n\ge2\}$, which are closed subspaces of $\Hs$.
                    \item Consider $m=(0,1/2,1/3,1/4,\dots)$ and $\Psi,\Phi$ to be Riesz sequences in $\Hs$. Similar arguments as in Remark \ref{multiplierwithnonclosedrange} show that $\ran(\mult)$ and $\ran(M_{m,\Psi,\Phi})$ are not closed in $\Hs$.       
                \end{enumerate}

            \end{enumerate}
        \end{ex}

    If $\Psi$ and $\Phi$ are both overcomplete frames, then $\mult$ can either be invertible (e.g., take $m=(1)$ and $\Psi=\Phi$) or noninvertible. In the case of noninvertibility, the multiplier may  have closed range  or nonclosed range, as illustrated in {Example \ref{exofnoninvwithclosedandnonclosedrange}} below. To show the closedness of the range of the {constructed multiplier in Example \ref{exofnoninvwithclosedandnonclosedrange}(1)}, we first mention the following lemma {and provide a short proof,  as no proof appears to be available in the existing literature, to the best of our knowledge. }

    \begin{lemma}\label{HisWoplusWpiffWclosed}
        Let $\Hs$ be a separable Hilbert space and let $W$ be a subspace of $\Hs$. Then $\Hs=W\oplus W^\perp$ if and only if $W$ is closed.
    \end{lemma}
    \begin{proof}
        The ``if'' part is well-known. Assume that $\Hs=W\oplus W^\perp$ and let $(h_n)_{n=1}^\infty $ be a sequence in $W$ that converges to some $h\in\Hs.$ Write $h=u+v$, where $u$ and $v$ are the unique vectors in $W$ and $W^\perp$, respectively. Then $(h_n-u)_{n=1}^\infty $ is a sequence in $W$, hence in $\overline{W}$, that converges to $v\in W^\perp$. As $\overline{W}$ is closed, $v$ must be 0, and thus $h=u\in W$, as desired.
    \end{proof}

    \begin{ex}\label{exofnoninvwithclosedandnonclosedrange}
        Consider a separable Hilbert space $\Hs$ and let $(e_n)_{n=1}^\infty $ be an orthonormal basis for $\Hs$.
        \begin{enumerate}[(1)]
            \item  If $m=(1),\Psi=(e_1,e_1,e_2,e_2,e_3,e_4,\dots)$, and $\Phi=(e_1,e_2,e_2,e_1,e_3,e_4,\dots),$ then the multiplier $\mult$ is noninvertible with closed range. Indeed, observe that  $\mult h=h+\langle h,e_1\rangle e_2+\langle h,e_2\rangle e_1$ for every $h\in\Hs$.  So, $\mult e_1=e_1+e_2=\mult e_2$  implies $\mult$ is not injective.  We also see that $e_1-e_2\in\ran(\mult)^\perp$. Moreover, for every $h\in\Hs$,
    \begin{align*}
        h &= \frac12\langle h,e_1+e_2\rangle (e_1+e_2)+\sum_{n\ge3}\langle h,e_n\rangle e_n+\frac12\langle h,e_1-e_2\rangle(e_1-e_2)
        \\&= \mult\left( \frac12\langle h,e_1+e_2\rangle e_1 +\sum_{n\ge3}\langle h,e_n\rangle e_n \right)+\frac12\langle h,e_1-e_2\rangle(e_1-e_2)
        \\&\in\ran(\mult)\oplus\ran(\mult)^\perp.
    \end{align*}
    This shows that $\Hs=\ran(\mult)\oplus\ran(\mult)^\perp$. Thus, $\mult$ has closed range by Lemma \ref{HisWoplusWpiffWclosed}. 
    
        \item  We construct a noninvertible multiplier with nonclosed range. Consider $m=(1),  \\\Psi=(e_1,e_1,e_2,e_2,e_3,e_3,e_4,e_4,\dots)$, and $\Phi=(e_1,e_2,e_2,e_3,e_3,e_4,e_4,\dots)$.
        Then for any $h\in\Hs$,
        \[ \mult h = h + \sum_{n=1}^\infty \langle h,e_n\rangle e_{n+1}. \]
        To show that $\ran(M)$ is dense in $\Hs$, first observe that $Me_n = e_n+e_{n+1}$ for all $n\in\N$ and therefore, $\spann\{e_n+e_{n+1}\}_{n=1}^\infty \subseteq\ran(\mult)$ implying $\overline{\spann}\{e_n+e_{n+1}\}_{n=1}^\infty \subseteq\overline\ran(\mult).$ But it was shown in \cite[Example 5.4.6]{OleIntroductiontoFramesandRieszBases} that $\overline{\spann}\{e_n+e_{n+1}\}_{n=1}^\infty =\Hs.$ Thus, $\overline\ran(\mult)=\Hs.$ But, clearly, $e_1\notin\ran(\mult)$ and therefore, $\ran(M)\neq\overline\ran(M).$
        \end{enumerate}
    
    \end{ex}

    The next result provides sufficient conditions to ensure the closedness of the range of a given frame multiplier.

    \begin{prop}\label{whendoesaframemultiplierhaveaclosedrange}
        Let $m$ be semi-normalized and let $\Psi$ and $\Phi$ be  frames {for $\Hs_1$ and $\Hs_2$, respectively}. Suppose that $\Ker(T_\Phi)\subseteq\Ker(T_{\overline m\Psi})$ or $\Ker(T_{m\Phi})\subseteq\Ker(T_\Psi)$. Then $\mult$ is injective and has a closed range in $\Hs_2$.
    \end{prop}
    \begin{proof}
        Denote $M:=\mult$ and let $(Mh_n)_{n=1}^\infty $ be a Cauchy sequence in $\Hs_2$. Assume first that $\Ker(T_\Phi)\subseteq\Ker(T_{\overline m\Psi})$.  Using the frame $\overline m\Psi$  and using \cite[Lemma 5.5.5]{OleIntroductiontoFramesandRieszBases}, we have $\ran(T^*\und{\overline m\Psi})=\overline{\ran}(T^*\und{\overline m\Psi})=\Ker(T\und{\overline{m}\Psi})^\perp\subseteq\Ker(T_\Phi)^\perp$, and  for any $n,\ell\in\N$,
        \begin{align*}
            \|Mh_n-Mh_\ell\|=\|T_\Phi T^*\und{\overline m\Psi}(h_n-h_\ell)\|\ge A_\Phi A_{\overline m\Psi}\|h_n-h_\ell\|.          
        \end{align*}
        Therefore, $(h_n)_{n=1}^\infty $ converges in $\Hs_1$. Consequently, the sequence $(Mh_n)_{n=1}^\infty $   converges to some element in $\ran(M)$, and so $\ran(M)$ is closed. The injectivity of $M$  immediately follows from the inequality $\|Mh\|\ge A_\Phi A_{\overline m\Psi}\|h\|$, $h\in\Hs_1$, obtained from arguments as above. 
        
        Assume now that $\Ker(T_{m\Phi})\subseteq\Ker(T_\Psi)$. Writing $\mult=T\und{m\Phi}T^*_\Psi$ and using similar arguments as above, we get the desired conclusions.
    \end{proof}

    \begin{rem} The two conditions in Proposition \ref{whendoesaframemultiplierhaveaclosedrange} are not equivalent. Indeed, consider a separable Hilbert space $\Hs$ and an orthonormal basis $(e_n)_{n=1}^\infty $ for $\Hs$. Then the sequences $m=(1,1/2,1,1,1,\dots)$, $\Phi=(e_1,e_1,e_2,e_3,e_4,\dots)$, and $\Psi=(e_1,2e_1,e_2,e_3,e_4,\dots)$ satisfy $\Ker(T_\Phi)\subseteq\Ker(T_{\overline m\Psi})$, but not \\$\Ker(T_{m\Phi})\subseteq\Ker(T_\Psi)$.
    \end{rem}
        
    Proposition \ref{whendoesaframemultiplierhaveaclosedrange} only talks about the closedness of the range of the multiplier $\mult$.  It is then a natural question if the closedness of $\ran(\mult)$ is equivalent to the closedness of $\ran(M_{m,\Psi,\Phi})$. Note that in some cases, this is true, but not in general. Indeed, if $m$ is a sequence of real numbers and $\mult$ is invertible, then $\ran(\mult)$ and $\ran(\mult^*)=\ran(M_{m,\Psi,\Phi})$ are both closed. On the other hand, the example below lists sequences $m,\Psi, \Phi$ where $\ran(\mult)$ is closed, but $\ran(M_{m,\Psi,\Phi})$ is not closed.

    \begin{ex}
        Let $(e_n)_{n=1}^\infty $ be an orthonormal basis for $\Hs$. Consider 
        $
            m=(1,\tfrac{i}{2},\tfrac12,1,\tfrac i2,\tfrac12,1,\tfrac i2,\tfrac12,\dots), \Psi=(e_1,e_2,e_2,e_2,e_3,e_3,e_3,e_4,e_4,e_4,\dots),
             \Phi=(e_1,ie_1,e_1,e_2,ie_2,e_2,e_3,ie_3,e_3,\dots).$
        Then $\mult = I_\Hs$, hence with closed range. Further for every $h\in\Hs$, we have that $M_{m,\Psi,\Phi}h=h+\sum_{n=1}^\infty \langle h,e_n\rangle e_{n+1}$, which from Example \ref{exofnoninvwithclosedandnonclosedrange}(2) shows that $M_{m,\Psi,\Phi}$  has nonclosed range.
    \end{ex}    

    \section*{Acknowledgements}

    This research was funded in whole by the Austrian Science Fund (FWF) [grant DOI 10.55776/P35846]. For open access purposes, the authors have applied a CC BY public copyright license to any author-accepted manuscript version arising from this submission.

	\bibliographystyle{plain}
	\bibliography{OblPsInvofMultbib}

@article{MilneOMP1968,
	author = {Milne, R. D.},
	title = "{An oblique matrix pseudoinverse}",
	journal = {SIAM Journal on Applied Mathematics},
	volume = {16},
	number = {5},
	pages = {931-944},
	year = {1968},
	doi = {10.1137/0116075},
	
	URL = { 
		
		https://doi.org/10.1137/0116075
		
		
		
	},
	eprint = { 
		
		https://doi.org/10.1137/0116075
		
		
		
	}
}

@article{EldarTobiasGeneralFrameworkforConsistentSamplinginHilbertSpaces,
author = {Eldar, Y.C. and Werther, T.},
year = {2005},
pages = {},
title = "{General framework for consistent sampling in Hilbert spaces}",
volume = {03},
number = {04},
journal = {International Journal of Wavelets, Multiresolution and Information Processing},
doi = {10.1142/S0219691305000981}
}

@book{KatoPerturbationTheoryforLinearOperators,
      author        = "Kato, T.",
      title         = "{Perturbation theory for linear operators; 2nd ed.}",
      publisher     = "Springer",
      address       = "Berlin",
      series        = "Grundlehren der mathematischen Wissenschaften : a series
                       of comprehensive studies in mathematics",
      year          = "1976",
      url           = "https://cds.cern.ch/record/101545",
}

@book{OleIntroductiontoFramesandRieszBases,
series = "{Applied and Numerical Harmonic Analysis}",
publisher = {Birkhäuser},
isbn = {0817642951},
year = {2003},
title = "{An Introduction to Frames and Riesz Bases}",
language = {eng},
address = {Boston, Mass.},
author = {Christensen, O.},
}

@article{BalazsBasicDefinitionandPropertiesofBesselMultipliers,
title = "{Basic definition and properties of Bessel multipliers}",
journal = {Journal of Mathematical Analysis and Applications},
volume = {325},
number = {1},
pages = {571-585},
year = {2007},
issn = {0022-247X},
doi = {https://doi.org/10.1016/j.jmaa.2006.02.012},
url = {https://www.sciencedirect.com/science/article/pii/S0022247X06001260},
author = {P. Balazs}
}

@article{HeinekenMorillasObliqueDualFusionFrames,
author = {S.B. Heineken and P.M. Morillas},
title = "{Oblique dual fusion frames}",
journal = {Numerical Functional Analysis and Optimization},
volume = {39},
number = {7},
pages = {800--824},
year = {2018},
publisher = {Taylor \& Francis},
doi = {10.1080/01630563.2017.1421555},


URL = { 
    
        https://doi.org/10.1080/01630563.2017.1421555
    
    

},
eprint = { 
    
        https://doi.org/10.1080/01630563.2017.1421555
    
    

}

}

@article{ChristensenEldarObliqueDualFramesandShiftInvariantSpaces,
title = {Oblique dual frames and shift-invariant spaces},
journal = {Applied and Computational Harmonic Analysis},
volume = {17},
number = {1},
pages = {48-68},
year = {2004},
note = {Special Issue: Frames in Harmonic Analysis, Part 1},
issn = {1063-5203},
doi = {https://doi.org/10.1016/j.acha.2003.12.003},
url = {https://www.sciencedirect.com/science/article/pii/S1063520304000338},
author = {O. Christensen and Y.C Eldar}
}

@article{RobinsonOntheGeneralizedInverseofanArbitraryLinearTransformation,
 ISSN = {00029890, 19300972},
 URL = {http://www.jstor.org/stable/2312140},
 author = {D.W. Robinson},
 journal = {The American Mathematical Monthly},
 number = {5},
 pages = {412--416},
 publisher = {[Taylor & Francis, Ltd., Mathematical Association of America]},
 title = "{On the generalized inverse of an arbitrary linear transformation}",
 urldate = {2025-07-16},
 volume = {69},
 year = {1962},
 note = {In: Classroom notes, pages 406-425},
}

@article{LangenhopOnGeneralizedInversesofMatrices,
 ISSN = {00361399},
 URL = {http://www.jstor.org/stable/2099162},
 author = {C.E. Langenhop},
 journal = {SIAM Journal on Applied Mathematics},
 number = {5},
 pages = {1239--1246},
 publisher = {Society for Industrial and Applied Mathematics},
 title = "{On generalized inverses of matrices}",
 urldate = {2025-07-16},
 volume = {15},
 year = {1967}
}

@article{ChipmanOnLeastSquareswithInsufficientObservations,
author = {J.S. Chipman},
title = "{On least squares with insufficient observations}",
journal = {Journal of the American Statistical Association},
volume = {59},
number = {308},
pages = {1078--1111},
year = {1964},
publisher = {ASA Website},
doi = {10.1080/01621459.1964.10480751},


URL = { 
    
    
        https://www.tandfonline.com/doi/abs/10.1080/01621459.1964.10480751
    

},
eprint = { 
    
    
        https://www.tandfonline.com/doi/pdf/10.1080/01621459.1964.10480751
    

}

}

@article{WardetalANoteontheObliqueMatrixPseudoinverse,
 ISSN = {00361399},
 URL = {http://www.jstor.org/stable/2099917},
 author = {J.F. Ward and T.L. Boullion and T.O. Lewis},
 journal = {SIAM Journal on Applied Mathematics},
 number = {2},
 pages = {173--175},
 publisher = {Society for Industrial and Applied Mathematics},
 title = "{A note on the oblique matrix pseudoinverse}",
 urldate = {2025-07-17},
 volume = {20},
 year = {1971}
}

@phdthesis{EldarQuantumSignalProcessing,
  title        = {Quantum Signal Processing},
  author       = {Y.C. Eldar},
  year         = {2001},
  school       = {Massachusetts Institute of Technology},
  type         = {PhD thesis}
}

@article{StoevaBalazsInvertibilityofMultipliers,
	title = "{Invertibility of multipliers}",
	journal = {Applied and Computational Harmonic Analysis},
	volume = {33},
	number = {2},
	pages = {292-299},
	year = {2012},
	issn = {1063-5203},
	doi = {https://doi.org/10.1016/j.acha.2011.11.001},
	url = {https://www.sciencedirect.com/science/article/pii/S1063520311001102},
	author = {D.T. Stoeva and P. Balazs}
}

@article{StoevaBalazsDetailedCharacterizationofConditionsfortheUnconditionalConvergenceandInvertibilityofMultipliers, title="{Detailed characterization of conditions for the unconditional convergence and invertibility of multipliers}", volume={12}, url={http://dx.doi.org/10.1007/bf03549563}, DOI={10.1007/bf03549563}, number={2–3}, journal={Sampling Theory in Signal and Image Processing}, publisher={Springer Science and Business Media LLC}, author={Stoeva, D.T. and Balazs, P.}, year={2013},  pages={87–125}, language={en} }

@article{StoevaBalazsASurveyontheUnconditionalConvergenceandtheInvertibilityofFrameMultiplierswithImplementation, 
title="{A survey on the unconditional convergence and the invertibility of frame multipliers with implementation}", url={http://dx.doi.org/10.1007/978-3-030-36291-1_6}, DOI={10.1007/978-3-030-36291-1_6}, journal={Applied and Numerical Harmonic Analysis}, publisher={Springer International Publishing}, author={Stoeva, D.T. and Balazs, P.}, year={2020}, pages={169–192}, language={en} }

@article{BalazsStoevaRepresentationofhtInverseofaFrameMultiplier,
title = "{Representation of the inverse of a frame multiplier}",
journal = {Journal of Mathematical Analysis and Applications},
volume = {422},
number = {2},
pages = {981-994},
year = {2015},
issn = {0022-247X},
doi = {https://doi.org/10.1016/j.jmaa.2014.09.020},
url = {https://www.sciencedirect.com/science/article/pii/S0022247X14008440},
author = {P. Balazs and D.T. Stoeva}
}

@article{StoevaBalazsOntheDualFrameInducedbyanInvertibleFrameMultiplier, title="{On the dual frame induced by an invertible frame multiplier}", volume={15}, url={http://dx.doi.org/10.1007/bf03549600}, DOI={10.1007/bf03549600}, number={1}, journal={Sampling Theory in Signal and Image Processing}, publisher={Springer Science and Business Media LLC}, author={Stoeva, D.T. and Balazs, P.}, year={2016},  pages={119–130}, language={en} }

@article{EldarSamplingwithArbitrarySamplingandReconstructionSpacesandObliqueDualFrameVectors, title="{Sampling with arbitrary sampling and reconstruction spaces and oblique dual frame vectors}", volume={9}, url={http://dx.doi.org/10.1007/s00041-003-0004-2}, DOI={10.1007/s00041-003-0004-2}, number={1}, journal={Journal of Fourier Analysis and Applications}, publisher={Springer Science and Business Media LLC}, author={Eldar, Y.C.}, year={2003},  pages={77–96} }

@article{DiazHeinekenMorillasApproximateobliqueDualFrames,
title = "{Approximate oblique dual frames}",
journal = {Applied Mathematics and Computation},
volume = {452},
pages = {128015},
year = {2023},
issn = {0096-3003},
doi = {https://doi.org/10.1016/j.amc.2023.128015},
url = {https://www.sciencedirect.com/science/article/pii/S0096300323001844},
author = {J.P. Díaz and S.B. Heineken and P.M. Morillas}
}

@article{DiazHeinekenMorillasObliqueDulalityforFusionFrames, 
title="{Approximate oblique duality for fusion frames}", url={https://arxiv.org/abs/2401.16539}, DOI={10.48550/ARXIV.2401.16539}, abstractNote={Fusion frames are a convenient tool in applications where we deal with a large amount of data or when a combination of local data is needed. Oblique dual fusion frames are suitable in situations where the analysis for the data and its subsequent synthesis have to be implemented in different subspaces of a Hilbert space. These procedures of analysis and synthesis are in general not exact, and also there are circumstances where the exact dual is not available or it is necessary to improve its properties. To resolve these questions we introduce the concept of approximate oblique dual fusion frame, and in particular of approximate oblique dual fusion frame system. We study their properties. We give the relation to approximate oblique dual frames. We provide methods for obtaining them. We show how to construct other duals from a given one that give reconstructions errors as small as we want.}, 
journal={arXiv},publisher={arXiv}, author={Díaz, J.P. and Heineken, S.B. and Morillas, P.M.}, year={2024} }

@conference{AceskaMalonzoVelascoODFcompletion,
  title        = "{Oblique dual frame completion}",
  author       = {R. Aceska and J.V. Malonzo and G.A.M. Velasco},
  year         = {2025},
  booktitle    = {Proceedings of the International Conference on Sampling Theory and Applications},
}

@article{XiaoZhuZengODFinFinDimHS,
author = {Xiao, X.C. and Zhu, Y.C. and Zeng, X.M.},
title = "{Oblique dual frames in finite-dimensional Hilbert spaces}",
journal = {International Journal of Wavelets, Multiresolution and Information Processing},
volume = {11},
number = {02},
pages = {1350011},
year = {2013},
doi = {10.1142/S0219691313500112},

URL = { 
    
        https://doi.org/10.1142/S0219691313500112
    
    

},
eprint = { 
    
        https://doi.org/10.1142/S0219691313500112
    
    

}
}

@article{KooLimExtensionofBesselSeqtoODFSeqeunces, title="{Extension of Bessel sequences to oblique dual frame sequences and the minimal projection}", volume={30}, url={http://dx.doi.org/10.13001/1081-3810.2902}, DOI={10.13001/1081-3810.2902}, abstractNote={<jats:p>An extension of two Bessel sequences to oblique dual frame sequences and its applications to shift-invariant spaces are considered. The best-known situation where this kind of extension is necessary is the construction of a pair of biorthogonal multiresolution analyses, where two generating sets whose shifts are only assumed to be Bessel sequences are given. This extension naturally leads to consideration of the âminimal projectionâ extending two closed subspaces. The existence or non-existence of the minimal projection is discussed.</jats:p>}, journal={The Electronic Journal of Linear Algebra}, publisher={University of Wyoming Libraries}, author={Koo, Y. and Lim, J.}, year={2015},  }

@article{CasazzaTheArtofFrameTheory, title="{The art of frame theory}", volume={4}, url={http://dx.doi.org/10.11650/twjm/1500407227}, DOI={10.11650/twjm/1500407227}, number={2}, journal={Taiwanese Journal of Mathematics}, publisher={The Mathematical Society of the Republic of China}, author={Casazza, P.G.}, year={2000} }

@article{DuffinSchaefferAclassofnonharmonicFourierSeries, title="{A class of nonharmonic Fourier series}", volume={72}, url={http://dx.doi.org/10.1090/S0002-9947-1952-0047179-6}, DOI={10.1090/s0002-9947-1952-0047179-6}, number={2}, journal={Transactions of the American Mathematical Society}, publisher={American Mathematical Society (AMS)}, author={Duffin, R.J. and Schaeffer, A.C.}, year={1952}, pages={341–366}, language={en} }

@article{Daubechies_Grossmann_Meyer_1986, title="{Painless nonorthogonal expansions}", volume={27}, url={http://dx.doi.org/10.1063/1.527388}, DOI={10.1063/1.527388},number={5}, journal={Journal of Mathematical Physics}, publisher={AIP Publishing}, author={Daubechies, I. and Grossmann, A. and Meyer, Y.}, year={1986},  pages={1271–1283}, language={en} }

@book{CasazzaKutyniokFiniteFrames, title="{Finite Frames:Theory and Applications}",
subtitle={Theory and Applications}, 
author={P. Casazza and G. Kutyniok},
url={http://dx.doi.org/10.1007/978-0-8176-8373-3}, DOI={10.1007/978-0-8176-8373-3}, journal={Applied and Numerical Harmonic Analysis}, publisher={Birkhäuser Boston}, year={2013}, language={en} }

@inbook{MatzHlawatschLinearTimeFreqFiltersOnLineAlgoandApps,
author = {Matz, G. and Hlawatsch, F.},
year = {2002},
pages = {205-272},
title = "{Linear time-frequency filters: on-line algorithms and applications}",
journal = {Applications in Time-Frequency Signal Processing},
publisher = {CRC Press}
}

@article{GoyalEtAlQuantizedFrameExpansionswithErasures,
title = "{Quantized frame expansions with erasures}",
journal = {Applied and Computational Harmonic Analysis},
volume = {10},
number = {3},
pages = {203-233},
year = {2001},
issn = {1063-5203},
doi = {https://doi.org/10.1006/acha.2000.0340},
url = {https://www.sciencedirect.com/science/article/pii/S1063520300903403},
author = {V.K. Goyal and J. Kovačević and J.A. Kelner}
}

@book{FoucartRauhutAMathematicalIntroductiontoCompressiveSensing, title="{A Mathematical Introduction to Compressive Sensing}", url={http://dx.doi.org/10.1007/978-0-8176-4948-7}, DOI={10.1007/978-0-8176-4948-7}, journal={Applied and Numerical Harmonic Analysis}, publisher={Springer New York}, author={Foucart, S. and Rauhut, H.}, year={2013}, language={en} }

@book{BenedettoFerreiraModernSamplingTheory, editor = {J.J. Benedetto and P.J.S.G. Ferreira},title={Modern Sampling Theory: Mathematics and Applications}, subtitle = {Mathematics and Applications}, url={http://dx.doi.org/10.1007/978-1-4612-0143-4}, DOI={10.1007/978-1-4612-0143-4}, publisher={Birkhäuser Boston}, year={2001}, language={en} }

@article{StoevaTauboeck,
    author = {D.T. Stoeva and G. Taub\"ock},
    title = "{On pseudo-inverses of frame multipliers}",
    journal = {Preprint},
    year = {2018}    
}

@ARTICLE{BehrensScharfSignalProcessingApplicationsofObliqueProjectionOperators,
  author={Behrens, R.T. and Scharf, L.L.},
  journal={IEEE Transactions on Signal Processing}, 
  title="{Signal processing applications of oblique projection operators}", 
  year={1994},
  volume={42},
  number={6},
  pages={1413-1424},
  doi={10.1109/78.286957}}

@article{GrevilleSolutionoftheMatrixEqnandRelationsbetweenObliqueandorthogonalProjectors,
author = {Greville, T.N.E.},
title = "{Solutions of the matrix equation XAX = X, and relations between oblique and orthogonal projectors}",
journal = {SIAM Journal on Applied Mathematics},
volume = {26},
number = {4},
pages = {828-832},
year = {1974},
doi = {10.1137/0126074},

URL = { 
    
        https://doi.org/10.1137/0126074
    
    

},
eprint = { 
    
        https://doi.org/10.1137/0126074
    
    

}
}

@Inbook{FeichtingerNowakAFirstSurveyofGaborMultipliers,
author="Feichtinger, H.G.
and Nowak, K.",
editor="Feichtinger, H.G.
and Strohmer, T.",
title="A first survey of Gabor multipliers",
bookTitle="Advances in Gabor Analysis",
year="2003",
publisher="Birkh{\"a}user Boston",
address="Boston, MA",
pages="99--128",
isbn="978-1-4612-0133-5",
doi="10.1007/978-1-4612-0133-5_5",
url="https://doi.org/10.1007/978-1-4612-0133-5_5"
}

@book{HeilABasisTheoryPrimer, title={A Basis Theory Primer}, url={http://dx.doi.org/10.1007/978-0-8176-4687-5}, DOI={10.1007/978-0-8176-4687-5}, journal={Applied and Numerical Harmonic Analysis}, publisher={Birkhäuser Boston}, author={Heil, C.}, year={2011}, language={en} }

@article{CorachMaestripieriWeightedGeneralizedInversesObliqueProjectionsandLeast-SquaresProblems,
author = {G. Corach and A. Maestripieri},
title = {Weighted Generalized Inverses, Oblique Projections, and Least-Squares Problems},
journal = {Numerical Functional Analysis and Optimization},
volume = {26},
number = {6},
pages = {659--673},
year = {2005},
publisher = {Taylor \& Francis},
doi = {10.1080/01630560500323083},


URL = { 
    
        https://doi.org/10.1080/01630560500323083
    
    

},
eprint = { 
    
        https://doi.org/10.1080/01630560500323083
    
    

}

}

@incollection{NashedVotrubaAUnifiedOperatorTheoryofGeneralizedInverses,
title = {A Unified Operator Theory of Generalized Inverses},
editor = {M. Zuhair Nashed},
booktitle = {Generalized Inverses and Applications},
publisher = {Academic Press},
pages = {1-109},
year = {1976},
isbn = {978-0-12-514250-2},
doi = {https://doi.org/10.1016/B978-0-12-514250-2.50005-6},
url = {https://www.sciencedirect.com/science/article/pii/B9780125142502500056},
author = {M. Zuhair Nashed and G.F. Votruba}
}

@article{TangObliqueProjectionsBiorthogonalRieszBasesAndMultiwaveletsInHilbertSpaces,
 ISSN = {00029939, 10886826},
 URL = {http://www.jstor.org/stable/119911},
 author = {W.S. Tang},
 journal = {Proceedings of the American Mathematical Society},
 number = {2},
 pages = {463--473},
 publisher = {American Mathematical Society},
 title = {Oblique Projections, Biorthogonal {R}iesz Bases and Multiwavelets in {H}ilbert Spaces},
 urldate = {2026-02-16},
 volume = {128},
 year = {2000}
}

@InProceedings{StoevaBalazsRieszBasesMultipliers,
author="Stoeva, D. T.
and Balazs, P.",
editor="Cepedello Boiso, Manuel
and Hedenmalm, H{\aa}kan
and Kaashoek, Marinus A.
and Montes Rodr{\'i}guez, Alfonso
and Treil, Sergei",
title="Riesz bases multipliers",
booktitle="Concrete Operators, Spectral Theory, Operators in Harmonic Analysis and Approximation",
year="2014",
publisher="Springer Basel",
address="Basel",
pages="475--482",
isbn="978-3-0348-0648-0"
}

@article{BalanEquivalenceRelationsandDistancesbetweenHilbertframes,
 ISSN = {00029939, 10886826},
 URL = {http://www.jstor.org/stable/119271},
 author = {R. Balan},
 journal = {Proceedings of the American Mathematical Society},
 number = {8},
 pages = {2353--2366},
 publisher = {American Mathematical Society},
 title = {Equivalence Relations and Distances between {H}ilbert Frames},
 urldate = {2026-04-23},
 volume = {127},
 year = {1999}
}

@incollection{BariBiorthogonalSystemsandBasesinHilbertspaces,
    author = {Bari, N.K.},
    title = {Biorthogonal systems and bases in {H}ilbert space},
    booktitle = {Mathematics.~Vol.~IV Uch. Zap. Mosk. Gos. Univ.},
    publisher = {Moscow University Press}, 
    series ={issue 148},
    year = {1951},
    pages = {69-107}
}

@article{BariFrenchPaper,
  title = {Sur les bases dans l'espace de {H}ilbert},
  author = {Bari, N.K.},
  journal = {(Dokl.) Acad Sci. URSS},
  volume = {54},
  pages = {379--382},
  year = {1946},
  language = {fr}
}
	
\end{document}